\documentclass{amsart}

\usepackage[english]{babel}
\usepackage[utf8]{inputenc}

\usepackage{graphicx}
\usepackage{color}
\usepackage{caption}
\usepackage{subcaption}

\usepackage{amsmath}
\usepackage{amsthm}
\usepackage{amssymb}
\usepackage{algorithm}

\usepackage[noend]{algpseudocode}
\usepackage{enumitem}
\usepackage{diagbox}

\usepackage{mathabx}
\usepackage{tikz}
\usetikzlibrary{graphs}
\usepackage{tkz-graph}
\usetikzlibrary{fit}
\usepackage{makecell}

\usepackage{hyperref}
\newtheorem{theorem}{Theorem}[section]
\newtheorem{lemma}[theorem]{Lemma}
\newtheorem{corollary}[theorem]{Corollary}
\newtheorem{proposition}[theorem]{Proposition}
\newtheorem{example}{Example}

\theoremstyle{definition}

\theoremstyle{remark}

\DeclareMathOperator*{\argmax}{argmax\,}

\numberwithin{equation}{section}

\newcommand{\abs}[1]{\lvert#1\rvert}

\calclayout

\title[Orthogonal approximations to symmetric tensors]{Optimal orthogonal approximations to symmetric tensors cannot always be chosen symmetric}

\author{Oscar Mickelin}
\address{Department of Mathematics, Massachusetts Institute of Technology, Massachusetts, USA}
\email{oscarmi@mit.edu}

\author{Sertac Karaman}
\address{Department of Aeronautics and Astronautics, Massachusetts Institute of Technology, Massachusetts, USA}
\email{sertac@mit.edu}

\begin{document}

\begin{abstract}
We study the problem of finding orthogonal low-rank approximations of symmetric tensors. In the case of matrices, the approximation is a truncated singular value decomposition which is then symmetric. Moreover, for rank-one approximations of tensors of any dimension, a classical result proven by Banach in 1938 shows that the optimal approximation can always be chosen to be symmetric. In contrast to these results, this article shows that the corresponding statement is no longer true for orthogonal approximations of higher rank. Specifically, for any of the four common notions of tensor orthogonality used in the literature, we show that optimal orthogonal approximations of rank greater than one cannot always be chosen to be symmetric.
\end{abstract}




\maketitle

\section{Introduction}
Given a tensor $T \in \mathbb{R}^{n_1 \times \ldots \times n_d}$ or $T \in \mathbb{C}^{n_1 \times \ldots \times n_d}$, it is a well-studied problem to search for a compressed approximation of $T$. Representing the approximation using the canonical decomposition, computing a low-rank (and therefore low-storage) approximation of the form $\sum_{k=1}^r \sigma_k \bigotimes_{j=1}^d  v_{kj}$ is a classical problem with great practical interest \cite{kolda2009tensor}.

It is well-known that the optimal rank-$r$ approximation problem in $\mathbb{F} = \mathbb{R}$ or $\mathbb{F} = \mathbb{C}$
\begin{equation}\label{eq:intro1}
\inf_{v_{kj}} \left\{ \| T - \sum_{k=1}^r \sigma_k \bigotimes_{j=1}^d  v_{kj} \| : v_{kj} \in \mathbb{F}^{n_j}, \sigma_k \in \mathbb{F}\right\},
\end{equation}
is in general ill-posed for $r > 1$ since the set of tensors of rank at most $r$ is not necessarily closed \cite{de2008tensor}. For $r=1$, the rank-one approximation problem is well-posed \cite{hackbusch2012tensor}, but in general NP-hard to solve for any $d \geq 3$ \cite{hillar2013most}. Nonetheless, in practical computations, suboptimal approximations are often good enough, and a variety of methods exist to compute rank-$r$ approximations for any $r \geq 1$; we refer to the review article by Kolda and Bader \cite{kolda2009tensor} for an overview of these methods and a longer discussion.

The nature of the approximation problem changes when imposing orthogonality conditions on the vectors $v_{kj}$. There are a number of natural notions of orthogonality of rank-one tensors in the literature, and we recall their definitions in Section~\ref{sec:notanddef}. When restricting the families $v_{kj}$ to be orthogonal under any of these notions, the corresponding optimal orthogonal rank-$r$ approximation problem in Equation~\eqref{eq:intro1} has been shown to be well-posed for any $r \geq 1$ \cite{chen2009tensor,sorensen2012canonical,wang2015orthogonal}.

In this article, we will consider the orthogonal approximation problem under the additional assumption that $T$ is a symmetric tensor. When $T$ is symmetric, it is in many applications natural to look also for symmetric approximations of $T$, for instance when attempting to recover a symmetric tensor corrupted by noise in independent component analysis \cite{comon1994independent} or latent variable models \cite{anandkumar2014tensor,mu2015successive,mu2017greedy}. In the matrix case $d=2$, the truncated singular value decomposition shows that an optimal approximation can always be chosen symmetric. Moreover, when $r=1$ and $T$ is a symmetric tensor, a classical result proven by Banach in 1938 \cite{banach1938homogene} and also rediscovered recently \cite{zhang2012cubic,zhang2012best,friedland2013best} shows that the optimizer of Equation~\eqref{eq:intro1} can be chosen to be symmetric, i.e., $v_{k1} = v_{k2} = \ldots = v_{kd}$, for all terms $k$. However, for $r>1$, it is not known in general if an optimal orthogonal approximation to $T$ can necessarily be chosen to be symmetric. This is both of theoretical and practical interest, since a symmetric approximation has fewer degrees of freedom and the optimization problem in Equation~\eqref{eq:intro1} can therefore typically be accelerated and more easily analyzed. Furthermore, computations with symmetric tensors are known to admit algorithms with beneficial stability properties and speed \cite{solomonik2016contracting,schatz2014exploiting}.
This article shows that optimal orthogonal approximations to symmetric tensors cannot be chosen symmetric, in general, for any of the notions of orthogonality used in the literature.

\subsection{Related work}
Numerical algorithms for computing orthogonal approximations have received substantial attention in recent years. This has produced several classes of approximation algorithms for general tensors as well as improved algorithms for special cases of tensors satisfying certain structural assumptions. Nie and Wang \cite{nie2014semidefinite} phrased the rank-one approximation problem as a semidefinite optimization problem which can be solved for tensors of moderate size. Friedland and Wang \cite{friedland2018spectral} presented an alternative way of computing the best rank-one approximation of a symmetric tensor, by finding the fixed points of the associated polynomial map. For higher-rank orthogonal approximations, the literature has used non-linear optimization techniques to produce algorithms converging to local minima of Equation~\eqref{eq:intro1}. Chen and Saad \cite{chen2009tensor} introduced a higher-order power method for computing completely orthogonal approximations, with convergence guarantees to a local minimum. S{\o}rensen et al. \cite{sorensen2012canonical} considered alternating least-squares based algorithms for partially orthogonal approximations, and Wang, Chu and Yu \cite{wang2015orthogonal} presented a higher-order power method for the same purpose, also proving convergence to a local minimum.

The convergence guarantees of the orthogonal approximation problem have been improved by imposing structural assumptions on the tensors under consideration. An important special case that guarantees also global convergence is when $T$ is completely orthogonally decomposable, i.e., when $T$ has an exact orthogonal (but unknown) decomposition $T = \sum_{k=1}^r \bigotimes_{j=1}^d v_{kj}$ with $v_{kj} \perp v_{k'j}$ for all $k \neq k'$. Optimal rank-$r$ orthogonal decompositions for these tensors can be computed by successively computing the optimal rank-one approximations, subtracting these from $T$ and iterating, in a deflation procedure \cite{zhang2001rank}. An optimal rank-one approximation of orthogonally decomposable tensors can be computed efficiently by the tensor power iteration method \cite{anandkumar2014tensor} and each term corresponds to a singular vector of the original tensor. This method has also been extended to nearly completely orthogonally decomposable tensors \cite{anandkumar2014tensor,mu2015successive,mu2017greedy}.

Lastly, the results of this article are related to (but distinct from) Comon's conjecture \cite{comon2008symmetric}, which asks if the rank and symmetric rank of a symmetric tensor coincide, i.e., if
\begin{equation}\label{eq:comon}
\min_r \left\{r: T = \sum_{k=1}^r \sigma_k \bigotimes_{j=1}^d v_{kj} \right\} = \min_r \left\{r: T = \sum_{k=1}^r \sigma_k \underbrace{v_k\otimes \ldots \otimes v_k}_{d} \right\},
\end{equation}
whenever $T$ is a symmetric tensor. This conjecture has been proven to be true in many special cases \cite{zhang2016comon,friedland2016remarks}, but is now known to not be true in general \cite{shitov2018counterexample}. The setting of this article can be seen as an extension of Comon's conjecture to the case of orthogonal approximations of a tensor. This extension differs from Comon's conjecture in two ways. Firstly, we consider approximations in Equation~\eqref{eq:intro1}, rather than the exact decompositions in Equation~\eqref{eq:comon}. Secondly, we impose orthogonality constraints on the terms $v_{kj}$ in Equation~\eqref{eq:intro1}. This drastically changes the nature of the problem, since for instance the approximation problem is ill-posed without orthogonality constraints and well-posed with orthogonality constraints. The status of Comon's conjecture therefore does not have any direct bearing on the extension we consider. Our results show that this extension is not true in general, for any of the common notions of orthogonality of tensors considered in the literature.

\subsection{Contributions}
We treat a number of theoretical and practical questions concerning the optimal orthogonal rank-$r$ approximation problem for symmetric tensors. Our main contributions treat the symmetry of orthogonal optimizers of Equation~\eqref{eq:intro1}. We show that the optimizer of Equation~\eqref{eq:intro1} cannot in general be chosen symmetric, for any of the notions of orthogonality that appear in the literature (see Section~\ref{sec:notanddef} for a definition of these different notions). However, we show that the optimal completely orthogonal rank-$r$ approximation of a symmetric tensor $T$ can be chosen to be symmetric when $n_1= \ldots = n_d = 2$, and that a stronger condition than complete orthogonality results in symmetric optimizers. We also show that the optimal partially orthogonal rank-$r$ approximation can be chosen to have terms that are separately symmetric under permutations of two disjoint partitions of the tensor dimensions when $n_1= \ldots = n_d = 2$. However, the optimizer cannot be taken symmetric in all tensor dimensions, in general. We also prove a pair of results on the structure of symmetric orthogonal, strongly orthogonal and partially orthogonal tensors, which may be of independent interest.

Along the way, we also present a number of additional ways in which the orthogonal approximation problem differs from the matrix case, and rank-one approximation of general tensors. Firstly, we show that the optimal completely orthogonal rank-$r$ approximation of $T$ \emph{cannot} in general be computed by successive deflations using the optimal rank-one approximations, for non-orthogonally decomposable tensors. This is an analogue of a result in \cite{stegeman2010subtracting} to the setting of orthogonal decompositions. We also provide examples that show that the terms in the optimal orthogonal rank-$r$ approximation are not necessarily the tensor singular vectors for $r \geq 2$. Secondly, we provide examples of symmetric tensors $T$ for which the optimal completely orthogonal rank-$2$ approximation coincides with the optimal rank-$3$ approximation, but without being equal to $T$. This is a situation that does not occur in the matrix case. We conclude by showing that it is in general NP-hard to calculate the optimal orthogonal, strongly orthogonal, partially orthogonal, and completely orthogonal rank-$r$ approximations of a tensor (symmetric or not), for any $r \geq 1$. This result is known for $r=1$ \cite{hillar2013most}, but the case $r>1$ has not appeared in the literature, to the best of our knowledge.

The remainder of the article is structured as follows. Section~\ref{sec:notanddef} presents our notation, followed by a few auxiliary results in section~\ref{sec:auxiliary}. Section~\ref{sec:mainresults} contains our main results on the existence of symmetric optimal approximations. Lastly, section~\ref{sec:NP} concludes with a short result on NP-hardness of orthogonal approximations for any $r$.
\section{Notation and definitions}\label{sec:notanddef}
We will state our definitions and results in terms of a base field $\mathbb{F}$, which we will exclusively take to be either $\mathbb{R}$ or $\mathbb{C}$. Scalars will therefore be taken from $\mathbb{F}$ and tensors from $\mathbb{F}^{n_1 \times \ldots \times n_d}$. In contrast to the matrix case $d=2$, notions such as rank are dependent on the choice of base field, meaning that e.g., a tensor with real-valued entries can have different ranks over $\mathbb{R}$ and $\mathbb{C}$ \cite{kolda2009tensor}. We choose to emphasize this in our notation by including the subscript $\mathbb{F}$ wherever the result depends on the base field.

The $k$th standard basis vector will be denoted by $e_k$. The Kronecker delta will be denoted by $\delta_{k,k'}$. We will write the $k$-fold tensor power of a vector $v\in \mathbb{F}^n$ by $v^{\otimes k} := \underbrace{v\otimes \ldots \otimes v}_{k}$. For tensors $S$ and $T$ in $\mathbb{F}^{n_1 \times \ldots \times n_d}$, we define the Frobenius (or Hilbert-Schmidt) inner product and norm by
\begin{equation}
\langle T, S\rangle := \sum_{i_1, \ldots , i_d=1}^{n_1, \ldots, n_d} T (i_1, \ldots, i_d) \overline{S (i_1, \ldots, i_d}),
\end{equation}
and $\|T\| := \sqrt{\langle T, T\rangle}$. Given a tensor $T$ in $\mathbb{F}^{n_1 \times \ldots \times n_d}$ and a matrix $A$ in $\mathbb{F}^{m \times n_k}$, we define the $k$-mode contraction of $T$ by $A$ as a tensor $T\times_k A \in \mathbb{F}^{n_1 \times \ldots \times n_{k-1}\times m \times n_{k+1} \times \ldots \times n_d}$, with
\begin{equation}
(T\times_k A)(i_1, \ldots , i_d) = \sum_{j_k = 1}^{n_k} T(i_1, \ldots , i_{k-1}, j_k, i_{k+a}, \ldots , i_k)A(i_k, j_k).
\end{equation}
In particular, when $v\in \mathbb{F}^{n_k}$ is a vector, we view the resulting tensor $T\times_k v$ as an element of the space $\mathbb{F}^{n_1\times \ldots \times n_{k-1} \times n_{k+1} \times \ldots \times n_d}$, by omitting the singleton dimension.

The group of all permutations on $d$ elements will be denoted by $S^d$. A tensor $T \in \mathbb{F}^{n\times \ldots \times n}$ is called symmetric if $T (i_1, \ldots , i_d) = T (i_{\varphi(1)}, \ldots , i_{\varphi(d)})$ for all permutations $\varphi \in S^d$, and we denote the set of all symmetric tensors by $S^d(\mathbb{F}^{n})$.

The spectral norm of a tensor (also known as the injective norm; see e.g., \cite{hackbusch2012tensor}) is defined as
\begin{equation}\label{eq:spectraldef}
\|T\|_{\sigma, \mathbb{F}} := \sup_{x_k} \Big\{\abs{\langle T, x_1\otimes \ldots \otimes x_d\rangle} : x_k\in \mathbb{F}^{n_k} , \|x_k\| = 1 \Big\}.
\end{equation}
The dual of the spectral norm is the tensor nuclear norm (also known as the projective norm \cite{hackbusch2012tensor}), defined by
\begin{equation}\label{eq:nucleardef}
\| T\|_{*, \mathbb{F}} := \inf_{r,\sigma_k, v_{kj}} \left\{ \sum_{k=1}^r \abs{\sigma_k} : T = \sum_{k=1}^r \sigma_k v_{k1}\otimes \ldots \otimes v_{kd}, \|v_{kj}\| = 1, r\in \mathbb{N} \right\}.
\end{equation}
For a symmetric tensor $T$, both the spectral and nuclear norm of $T$ are achieved for symmetric maximizers, i.e., for $x_1 = x_2 = \ldots = x_d$ in Eq.~\eqref{eq:spectraldef} and $v_{k1} = v_{k2} = \ldots = v_{kd}$ in Eq.~\eqref{eq:nucleardef} \cite{banach1938homogene,friedland2018nuclear}.

For tensors in dimensions $d \geq 2$, we will use several different notions of orthogonality of two rank-one tensors. The following definitions were introduced by Kolda \cite{kolda2001orthogonal}. Let $x = x_1 \otimes \ldots \otimes x_d$ and $y = y_1 \otimes \ldots \otimes y_d$ be two tensors of rank $1$. We will say that $x$ and $y$ are 
\begin{itemize}
\item \emph{orthogonal} ($x\perp y$) if $\langle x, y \rangle = \langle x_1, y_1\rangle \cdot \ldots \cdot \langle x_d, y_d \rangle = 0$.
\item \emph{strongly orthogonal} ($x \perp_s y$) if $x \perp y$ and if for each $j = 1, \ldots , d$, either $x_{j} \perp y_{j}$ or $x_{j} = \mu_j y_{j}$ for some non-zero $0 \neq \mu_j \in \mathbb{F}$.
\item \emph{completely orthogonal} ($x\perp_c y$) if $x_{j} \perp y_{j}$ for all $j = 1, \ldots , d$. 
\end{itemize}
It is clear that complete orthogonality implies strong orthogonality, which in turn implies orthogonality. Let $ P \subseteq \{1, \ldots , d\}$ be a non-empty subset of tensor dimensions. We will also say that $x$ and $y$ are \emph{$P$-partially orthogonal} ($x \perp_P y$) if $\langle x_j, y_j\rangle = 0$ for each $j \in P$. The case when $\abs{P} = 1$ is also known as semiorthogonality in the literature \cite{sorensen2012canonical,wang2015orthogonal}.

Let now $T$ be a given tensor. For each notion of orthogonality, we will be interested in a decomposition of the form
\begin{equation}\label{eq:orth_decomp}
T = \sum_{k=1}^r \sigma_i  v_{k1}\otimes \ldots \otimes v_{kd}.
\end{equation}
A decomposition as in Equation~\eqref{eq:orth_decomp} is called an \emph{orthogonal, strongly orthogonal, completely orthogonal or partially orthogonal decomposition of $T$ with rank at most $r$} if for each pair $k\neq k'$, the terms $v_{k1}\otimes \ldots \otimes v_{kd}$ and $v_{k'1}\otimes \ldots \otimes v_{k'd}$ are orthogonal, strongly orthogonal, completely orthogonal or partially orthogonal, respectively. The set of all orthogonal, strongly orthogonal, completely orthogonal and partially orthogonal decomposition tensors of rank at most $r$ will be denoted by $\mathcal{ON}_r, \mathcal{SON}_r$, $\mathcal{CON}_r$, $\mathcal{PCON}_{r,P}$, respectively. The dimensions of the tensors will be clear from the context and therefore omitted from the notation. By a slight abuse of notation, we will write both $T \in \mathcal{ON}_r$ and $\{ v_{kj} \}_{k,j} \in \mathcal{ON}_r$ and likewise for $\mathcal{SON}_r, \mathcal{CON}_r$, and $\mathcal{PCON}_{r,P}$.

Note that by fixing an orthogonal basis $\{ v_{kj}\}_{k=1}^{n_j}$ of $\mathbb{F}^{n_j}$ for each $j$ shows that strongly orthogonal and orthogonal decompositions exist for any tensor, and for the matrix case $d=2$, the singular value decomposition guarantees the existence of a completely orthogonal decomposition of any matrix $M$. Importantly, however, for $d>2$, the existence of a completely orthogonal decomposition is a very special property and is not guaranteed for all tensors $T$ \cite{kolda2001orthogonal,zhang2001rank}.

\section{Orthogonal tensor approximations}\label{sec:auxiliary}
This section presents a few auxiliary results that are used in the remainder of the article. For a given tensor $T$ in $\mathbb{F}^{n_1 \times \ldots \times n_d}$, we consider the four problems of finding the optimal  orthogonal, strongly orthogonal, completely orthogonal and $P$-partially orthogonal rank-$r$ approximations $Y = \sum_{k=1}^r \sigma_k \bigotimes_{j=1}^d v_{kj}$ to $T$, i.e., of calculating
\begin{equation}
\min_{v_{kj}} \left\{ \| T - \sum_{k=1}^r \sigma_k \bigotimes_{j=1}^d v_{kj}\| :  \{v_{kj}\}_{k,j} \in \mathcal{A}_r, \|v_{kj}\| = 1 \right\}, 
\end{equation}
where $\mathcal{A}_r = \mathcal{ON}_r$, $ \mathcal{SON}_r$, $ \mathcal{CON}_r$ or $ \mathcal{PCON}_{r,P}$. For $d = 2$ or for completely orthogonally decomposable tensors, these three problems coincide, but they are in general distinct for tensors which are not completely orthogonally decomposable in $d > 2$. For the case of completely orthogonal tensors, Chen and Saad showed in \cite{chen2009tensor} that this problem is equivalent to 
\begin{equation}\label{eq:max1}
\max_{v_{kj}} \left\{ \sum_{k=1}^r\abs{ \langle T, \bigotimes_{j=1}^d v_{kj} \rangle}^2:  \{v_{kj}\}_{k,j} \in \mathcal{A}_r, \|v_{kj}\| = 1  \right\},
\end{equation}
and the proof carries through also to the other notions of orthogonality. In stark contrast to the case without any orthogonality assumptions on the $v_{kj}$, the domain of the problem is compact, so the maximum is in fact achieved and the problem is well-posed for any $r \geq 1$, although the maximizer is not necessarily unique (even for $d=2$).

For the discussion in this article, we will also frequently use the following alternative characterization of Equation~\eqref{eq:max1}.
\begin{proposition}\label{prop:max2}
For $\mathcal{A}_r = \mathcal{ON}_r$, $ \mathcal{SON}_r$, $ \mathcal{CON}_r$ or $ \mathcal{PCON}_{r,P}$, we have

\begin{equation}\label{eq:max2}
\begin{split}
\max_{v_{kj}} \left\{ \sum_{k=1}^r\abs{ \langle T, \bigotimes_{j=1}^d v_{kj} \rangle}^2:  \{v_{kj}\}_{k,j} \in \mathcal{A}_r, \|v_{kj}\| = 1  \right\} \\
= \left(\max_Y \left\{ \abs{\langle T, Y \rangle}:  Y \in \mathcal{A}_r, \|Y\| \leq 1  \right\} \right)^2.
\end{split}
\end{equation}

Moreover, if $Y = \sum_{k=1}^r \sigma_k \bigotimes_{j=1}^d v_{kj}$ with $\|v_{kj}\|= 1$ is a maximizer of the right hand side of Equation~\eqref{eq:max2}, then $v_{kj}$ is a maximizer of Equation~\eqref{eq:max1} and vice versa.
\end{proposition}

\begin{proof}
Let $Y = \sum_{k=1}^r \sigma_k \bigotimes_{j=1}^d v_{kj}$ be a (strongly, completely, partially) orthogonal decomposition of $Y$ with $ \sum_{k=1}^r \abs{\sigma_k}^2 \leq1$, and $\|v_{kj}\|= 1$. Cauchy-Schwarz gives
\begin{equation}
\abs{ \langle T, Y \rangle } \leq \sum_{k=1}^r \abs{ \sigma_k \langle T, \bigotimes_{j=1}^d v_{kj}\rangle} \leq \sqrt{\sum_{k=1}^r \abs{\langle T, \bigotimes_{j=1}^d v_{kj}\rangle}^2},
\end{equation} 
and equality is achieved when the $\sigma_k$ are proportional to $ \langle T, \bigotimes_{j=1}^d v_{kj}\rangle$.
\end{proof}

The expressions
\begin{equation}
\max_Y \left\{ \abs{\langle T, Y \rangle}:  Y \in \mathcal{A}_r, \|Y\| \leq 1  \right\}
\end{equation}
clearly define four different norms, which we will denote by $\| T\|_{\mathcal{ON}_r, \mathbb{F}}$, $\| T\|_{\mathcal{SON}_r, \mathbb{F}}$, $\| T\|_{\mathcal{CON}_r, \mathbb{F}}$, and $\| T\|_{\mathcal{PCON}_{r,P}, \mathbb{F}}$, respectively. For $r =1$, all four expressions coincide with the spectral tensor norm $\|T\|_{\sigma, \mathbb{F}}$. For a tensor $T$ with real-valued entries, it is known that the value of the spectral norm depends on if the tensor is seen as having base field $\mathbb{R}$ or $\mathbb{C}$, i.e., $\|T\|_{\sigma, \mathbb{R}} \neq \|T\|_{\sigma, \mathbb{C}}$ in general. The analogous statements $\| T\|_{\mathcal{ON}_r, \mathbb{R}} \neq \| T\|_{\mathcal{ON}_r, \mathbb{C}}$, $\| T\|_{\mathcal{SON}_r, \mathbb{R}} \neq \| T\|_{\mathcal{SON}_r, \mathbb{C}}$, $\| T\|_{\mathcal{CON}_r, \mathbb{R}} \neq \| T\|_{\mathcal{CON}_r, \mathbb{C}}$, $\| T\|_{\mathcal{PCON}_{r,P}, \mathbb{R}} \neq \| T\|_{\mathcal{PCON}_{r,P}, \mathbb{C}}$, in general, are also true for any $r \geq 1$, in light of Proposition~\ref{prop:blockr} below.

For the sake of completeness, we note the following result, which is an analogue of a result for the case $r=1$ presented in \cite{lim2014blind}.
\begin{proposition}
For $\mathcal{A}_r = \mathcal{ON}_r$, $ \mathcal{SON}_r$, $ \mathcal{CON}_r$ or $ \mathcal{PCON}_{r,P}$ and any $T$ in $\mathbb{F}^{n_1 \times \ldots \times n_d}$, the following inequalities hold
\begin{align}
\|T\|_{\sigma, \mathbb{F}} = \|T\|_{\mathcal{A}_1, \mathbb{F}} &\leq \|T\|_{\mathcal{A}_2, \mathbb{F}} \leq \ldots \leq \|T\| 
\leq  \ldots \leq  \|T\|_{\mathcal{A}_1, \mathbb{F}}^* = \|T\|_{*,\mathbb{F}}, \\
 &\|T\|_{\mathcal{CON}_r,\mathbb{F}} \leq \|T\|_{\mathcal{SON}_r,\mathbb{F}} \leq \|T\|_{\mathcal{ON}_r,\mathbb{F}}. 
\end{align}
Moreover, the dual norm $\|T\|_{\mathcal{A}_r,\mathbb{F}}^*$ can be characterized as
\begin{equation}\label{eq:dualONr}
\|T\|_{\mathcal{A}_r, \mathbb{F}}^* = \inf_{N,v_k} \left\{ \sum_{k = 1}^N \|v_k\| : T = \sum_{k=1}^N v_k, v_k \in \mathcal{A}_r \right\},
\end{equation}
and for $Y$ in $\mathcal{A}_r$ with corresponding (strongly, completely, partially) orthogonal decomposition $Y = \sum_{k=1}^r \sigma_k \bigotimes_{j=1}^d v_{kj}$ where $\|v_{kj}\| = 1$, it holds that $\|Y\|_{\mathcal{A}_r,\mathbb{F}} = \|Y\|_{\mathcal{A}_r,\mathbb{F}}^* = \|Y\| = \sqrt{\sum_{k=1}^r \sigma_k^2 }$.
\end{proposition}
\begin{proof}
The statements $\|T\|_{\mathcal{A}_k, \mathbb{F}} \leq \|T\|_{\mathcal{A}_{k+1}, \mathbb{F}}$ are clear by definition. To show that $\|T\|_{\mathcal{A}_k, \mathbb{F}} \leq \|T\|$, note that $\abs{\langle T, Y \rangle} \leq \|T\|$ for $\|Y\| \leq 1$, using Cauchy-Schwarz. The remaining inequalities follow by duality and the fact that $\|\cdot\|$ is self-dual. The second set of inequalities is clear from their definitions.

The remaining statements can be proven by exactly the same argument as in the case $r=1$ in \cite[Lemma 21]{lim2014blind}.
\end{proof}

\section{Symmetric approximations to symmetric tensors}\label{sec:mainresults}
This section contains our main results. For the remainder of the section, we let $T$ be a symmetric tensor in $S^d(\mathbb{F}^{n})$. The following theorem was proven by Banach \cite{banach1938homogene} and also rediscovered recently \cite{zhang2012cubic,zhang2012best,friedland2013best}. It shows that the optimal rank-one approximation of $T$ can in this case be chosen symmetric.
\begin{theorem}[\cite{banach1938homogene}]\label{thm:banach}
If $T \in S^d(\mathbb{F}^{n})$ is symmetric, then
\begin{equation}
\max_{\|x_k\| \leq 1} \abs{ \langle T, x_1 \otimes \ldots \otimes x_d \rangle }= \max_{\|x\| \leq 1} \abs{ \langle T, x^{\otimes d}\rangle}
\end{equation}
\end{theorem}

The remainder of the article is devoted to exploring extensions of this result to $r \geq 1$, while imposing one of our four different notions of orthogonality. Somewhat surprisingly and in contrast to the matrix case $d=2$, none of the notions of orthogonality result in the existence of symmetric global maximizers, in general. An overview of the results is provided in Table~\ref{tab:summary}.

\begin{table}[tbhp]
{\footnotesize
\caption{Summary of results in Section~\ref{sec:mainresults}.}\label{tab:summary}
\begin{center}
\begin{tabular}{c|c}
\textbf{Orthogonality} & \textbf{Can optimal approximations always be chosen symmetric?} \\ \hline
$\mathcal{ON}_r$ & In general, no (Thm.~\ref{thm:no_ON}) \\ \hline
$\mathcal{SON}_r$ & In general, no (Thm.~\ref{thm:no_SON}, Thm.~\ref{thm:no_ON}) \\ \hline
$\mathcal{CON}_r$ & Yes for $n=2$ (Thm.~\ref{thm:mainn2}), no for $n > 2$ (Thm.~\ref{thm:main})\\ \hline
$\mathcal{PCON}_{r,P}$ & \makecell{In general, no (Thm.~\ref{thm:no_ON}). \\
Separate symmetry of dimensions in $P$ for $n=2$ (Thm.~\ref{thm:mainn2partial}) \\
Symmetry of tensor dimensions in $\{1, \ldots , d\} \smallsetminus P$ (Thm.~\ref{theorem:mainpartial})} \\
 \hline
\end{tabular}
\end{center}
}
\end{table}

The proofs of these statements are given in section~\ref{sec:main_proofs}. These will require a few results on the structure of symmetric tensors under orthogonality constraints, given in section \ref{sec:struct}, as well as a semidefinite programming formulation of the orthogonal approximation problem, provided in section~\ref{sec:SDP}. In addition, section~\ref{sec:examples} contains some further examples of how orthogonal approximations in the general tensor case differ from the matrix case, and section~\ref{sec:NP} concludes by showing that the approximation problem is in general NP-hard, for any $r\geq 1$.

\subsection{Symmetric tensors under orthogonality constraints}\label{sec:struct}
This section contains a number of structural results that are used to prove the main results in Table~\ref{tab:summary}. We would first like to point out the following distinction between symmetric tensors and symmetric decompositions of a tensor. For a tensor $Y \in \mathcal{A}_r$, one could ask for two seemingly different notions of symmetry: (i) for $Y$ to be symmetric with rank no more than $r$, or (ii) for the seemingly stronger condition that $Y$ has a symmetric decomposition of the form $Y = \sum_{k=1}^r \sigma_k v_k^{\otimes d}$. Without imposing any orthogonality conditions, the question of whether or not the sets in (i) and (ii) are equal is known in the literature as Comon's conjecture \cite{comon2008symmetric}, which has been proven in many special cases \cite{zhang2016comon,friedland2016remarks}, but is now known to not be true in general \cite{shitov2018counterexample}. For $Y \in \mathcal{CON}_r$, these two notions are however equivalent, because the terms in a rank decomposition of an orthogonally decomposable tensor can be uniquely computed by successively computing the optimal rank-one deflations \cite{zhang2001rank}, i.e., by recursively defining $Y_0 = 0$, $\sigma_i = \langle Y, Y_i \rangle$ and
\begin{equation}\label{eq:deflate}
Y_{i+1} := y_{i+1,1}\otimes \ldots \otimes y_{i+1, d} =  \argmax_{\|y_j\| \leq 1} \abs{ \langle Y - \sum_{k=1}^{i} \sigma_kY_k, y_1 \otimes \ldots \otimes y_d \rangle}.
\end{equation}
By Theorem~\ref{thm:banach} and the uniqueness of rank$-1$ approximations, when $Y$ is symmetric, each $Y_i$ is symmetric as well, i.e., $Y_i = v_i^{\otimes d}$. We now show that this statement also holds for symmetric tensors in $\mathcal{PCON}_{r,P}$, for any non-empty subset $P \subseteq \{1, \ldots , d\}$, i.e., when imposing partial orthogonality, the resulting analogue of Comon's conjecture is true. We will need the following result:

\begin{lemma}\label{lemma:symoperations}
For a symmetric tensor $T \in S^d(\mathbb{F}^n)$ and any vector $v \in \mathbb{F}^n$
\begin{enumerate}
\item $T\times_j v \in S^{d-1}(\mathbb{F}^n)$ is a symmetric tensor for any index $j$.
\item $T \times_j v = T \times_k v$ for any indices $1 \leq j,k \leq d$.
\end{enumerate}
\end{lemma}
\begin{proof}
For the first statement, let $\varphi \in S^{d-1}$ be any permutation on $d-1$ elements. We have
\begin{equation}
\begin{split}
(T \times_j v)&(i_{\varphi(1)}, \ldots , i_{\varphi(j-1)}, i_{\varphi(j+1)}, \ldots, i_{\varphi(d)}) =\\
&= \sum_{i_j =1}^n T (i_{\varphi(1)},  \ldots, i_j, \ldots, i_{\varphi(d)}) v_{i_j} \\
&= \sum_{i_j =1}^n T (i_{1},  \ldots, i_{d}) v_{i_j} =  (T \times_j v)(i_1, \ldots , i_{j-1}, i_{j+1}, \ldots, i_d),
\end{split}
\end{equation}
where the second equality comes from $T$ being a symmetric tensor.

For the second statement, let $\varphi \in S^d$ be the permutation of $(1, \ldots , d)$ that swaps $j$ and $k$ and leaves the other indices unchanged. We can assume $j \leq k$ for notational purposes, since the complementary case follows by relabeling $k \leftrightarrow j$. The symmetry of $T$ implies that
\begin{equation}
\begin{split}
(T \times_j v)(i_1, \ldots , i_{j-1}, & i_{j+1}, \ldots, i_d) = \sum_{i_j =1}^n T (i_1,  \ldots, i_d) v_{i_j} \\
&= \sum_{i_j =1}^n T (i_{\varphi(1)},  \ldots, i_{\varphi(d)}) v_{i_j} \\
&=  (T \times_k v)(i_1, \ldots , i_{j-1}, i_k, i_{j+1}, \ldots, i_{k-1}, i_{k+1}, \ldots,  i_d) \\
&= (T \times_k v)(i_1, \ldots , i_{j-1}, i_{j+1}, \ldots,  i_d),
\end{split}
\end{equation}
where the second equality follows from the symmetry of $T$ and the last equality by the first statement of the Lemma.
\end{proof}

We next prove the first main result of this section, on the structure of symmetric and partially orthogonal tensors.
\begin{proposition}\label{prop:symdecomp}
Take $d\geq 3$ and $v_{kj}$ vectors with $\|v_{kj}\| = 1$. Let $T \in S^d(\mathbb{F}^n)$ be a symmetric tensor with decomposition $T = \sum_{k=1}^r \sigma_k \bigotimes_{j=1}^d v_{kj}$. Assume that there is an index $1 \leq j_* \leq d$ such that $v_{kj_*} \perp v_{k'j_*}$ for all $k \neq k'$. If $r$ is the minimal integer for which such a decomposition exists, then $v_{kj} = v_{k1}$ up to multiplication by a complex phase factor, for all $1 \leq j \leq d$ and $1 \leq k \leq r$.

Moreover, if there are two distinct indices $1 \leq j_*, j_{**} \leq d$ with $v_{kj_*} \perp v_{k'j_*}$ and also $v_{kj_{**}} \perp v_{k'j_{**}}$ for all $k \neq k'$, then any such decomposition has minimal $r$, which also equals the rank of $T$.
\end{proposition}
\begin{proof}
By permuting the tensor dimensions if necessary, we can without loss of generality assume that $j_* = 1$. For any $i$, $T\times_1 v_{i1} = \sigma_i \otimes_{j=2}^d v_{ij}$ is symmetric by Lemma~\ref{lemma:symoperations}, so $v_{ij} = v_{i2}$ up to multiplication by a complex phase factor, for each $j \geq 2$, i.e., $T =  \sum_{k=1}^r \sigma_k v_{k1} \otimes v_{k2}^{\otimes d-1}$ after absorbing the complex phase factors into $\sigma_k$.

We next prove that $v_{i1} = v_{i2}$ up to a complex phase factor. By Lemma~\ref{lemma:symoperations}, we have
\begin{equation}
T \times_1 v_{i1} = \sigma_i v_{i2}^{\otimes d-1} = T \times_2 v_{i1} =\sum_{k=1}^r \sigma_k (v_{i1}\cdot v_{k2})v_{k1} \otimes v_{k2}^{\otimes d-2}.
\end{equation}
For any $j$, acting with $v_{j1}$ on the first tensor dimension on both sides of this equation, we obtain $ \sigma_i (v_{j1}\cdot v_{i2}) v_{i2}^{\otimes d-2} = \sigma_j (v_{i1}\cdot v_{j2}) v_{j2}^{\otimes d-2}$. This implies that, for each $j$, either (i) $v_{i2} = v_{j2}$ up to a complex phase factor, or (ii) $v_{j1} \perp v_{i2}$ and $v_{i1} \perp v_{j2}$.

If the first case holds, i.e., if for some $j\neq i$ and complex $\lambda$, $v_{i2} = \lambda v_{j2}$, then
\begin{equation}
\sigma_i v_{i1}\otimes v_{i2}^{d-1} + \sigma_j v_{j1}\otimes v_{j2}^{d-1} = (\sigma_i\lambda^{d-1}v_{i1} + \sigma_jv_{j1})\otimes v_{j2}^{\otimes d-1},
\end{equation}
so $T =  (\sigma_i\lambda^{d-1}v_{i1} + \sigma_jv_{j1})\otimes v_{j2}^{\otimes d-1} + \sum_{k\neq i,j}^r \sigma_k v_{k1} \otimes v_{k2}^{\otimes d-1}$ is a strictly shorter decomposition of $T$ with orthogonality in the first tensor dimension, which contradicts the minimality of $r$.

We have therefore shown that $v_{i1} \perp v_{j2}$ and $v_{j1} \perp v_{i2}$ for all $j \neq i$. This implies that $T\times_1 v_{i1} = \sigma_i v_{i2}^{\otimes d-1} = T\times_2 v_{i1} = \sigma_i (v_{i1}\cdot v_{i2}) v_{i1}\otimes v_{i2}^{\otimes d-2}$. The left hand side shows that $v_{i1} \cdot v_{i2} \neq 0$, so $v_{i1} = v_{i2}$ up to a complex phase factor. Since $i$ was arbitrary, this concludes the proof of the first statement, after absorbing the complex phase factor into $\sigma_i$.

For the second statement, we can assume $j_* = 1, j_{**} = 2$ by permuting the tensor dimensions, if necessary. Let $M$ be the first unfolding matrix of $T$ defined by $M\bigl(i_1, (i_2\ldots i_d)\bigr) = T(i_1 \ldots i_d)$ with $(i_2\ldots i_d)$ written as one long index. $M$ then has the decomposition $M = \sum_{k=1}^r v_{k1}\otimes u_k$ with $u_k(i_2\ldots i_d) = \prod_{j=2}^d v_{kj}(i_j)$, so $\text{rank}(M) \leq r$. Now, $v_{k1} \perp v_{k'1}$ for $k \neq k'$ and $u_k \perp u_{k'}$, since $v_{k2} \perp v_{k'2}$. It follows that $ r = \text{rank}(M)$. Since $r \geq \text{rank}(T) \geq \text{rank}(M)  = r$, we conclude that $r = \text{rank}(T)$ and we showed the second statement.
\end{proof}

We can extend the previous result also to the following setting.
\begin{corollary}\label{cor:symdecomp}
Take $d\geq 3$ and $v_{kj}$ vectors with $\|v_{kj}\| = 1$. Let $T \in S^d(\mathbb{F}^n)$ be a symmetric tensor with decomposition $T = \sum_{k=1}^r \sigma_k \bigotimes_{j=1}^d v_{kj}$ where there is an index $1 \leq j_* \leq d$ so that the vectors $\{v_{kj_*}\}_{k=1}^r$ are linearly independent. If $r$ is the minimal integer for which such a decomposition exists, then $v_{kj} = v_{k1}$ up to multiplication by a complex phase factor, for all $1 \leq j \leq d$ and $1 \leq k \leq r$.
\end{corollary}
\begin{proof}
Since the vectors $\{v_{kj_*}\}_{k=1}^r$ are linearly independent, there is an invertible matrix $A$ mapping each $v_{kj_*}$ to $e_k$. The tensor $S = T\times_1 A \times_2 \ldots \times_d A$ is symmetric and can be written as $S = \sum_{k=1}^r \sigma_k \bigotimes_{j=1}^d Av_{kj}$. $S$ is therefore in $\mathcal{PCON}_{r, \{j_*\}}$. Moreover, if there is an $s < r$ such that $S = \sum_{k=1}^s \lambda_k \bigotimes_{j=1}^d w_{kj}$ with the vectors $\{w_{kj_*}\}_{k=1}^s$ mutually orthogonal, then the vectors $\{A^{-1}w_{kj_*}\}_{k=1}^s$ are linearly independent, and $T = S\times_1 A^{-1} \times_2 \ldots \times_d A^{-1} = \sum_{k=1}^s \lambda_k \bigotimes_{j=1}^d A^{-1}w_{kj}$, which contradicts the minimality of $r$. Applying Theorem~\ref{prop:symdecomp} now shows that $Av_{kj} = Av_{k1}$, for all $k$ and $j$, so also $v_{kj} = v_{k1}$.
\end{proof}

In the case $d = 2$, the first statement in Proposition~\ref{prop:symdecomp} is no longer true, as shown by the decomposition $\bigl[ \begin{smallmatrix} 1 & 1 \\ 1 & 0 \end{smallmatrix} \bigr] = \bigl[ \begin{smallmatrix} 1\\ 1 \end{smallmatrix} \bigr] \otimes \bigl[ \begin{smallmatrix} 1\\ 0 \end{smallmatrix} \bigr]  + \bigl[ \begin{smallmatrix} 1\\ 0 \end{smallmatrix} \bigr] \otimes \bigl[ \begin{smallmatrix} 0\\ 1 \end{smallmatrix} \bigr]$, which has minimal length since the matrix has rank $2$ and has orthogonality in the second tensor dimension.

The corresponding statement of Proposition~\ref{prop:symdecomp} for tensors $T \in \mathcal{SON}_r$ or $ \mathcal{ON}_r$ is however not true. We must therefore in general distinguish between symmetric approximations and approximations with symmetric decompositions. In fact, we will use the following two characterizations of symmetric tensors in $\mathcal{ON}_r$ and $\mathcal{SON}_r$ for $r =2$ and $r=3$, respectively.

\begin{theorem}\label{thm:symrank2}
For any $n$ and $d$, we have
\begin{equation}
\begin{split}
\mathcal{ON}_2\cap S^d(\mathbb{F}^n) &= \mathcal{SON}_2\cap S^d(\mathbb{F}^n) = \mathcal{CON}_2\cap S^d(\mathbb{F}^n) \\
&= \{\sigma_1 v_1^{\otimes d} + \sigma_2 v_2^{\otimes d} : \sigma_1, \sigma_2 \in \mathbb{F}, v_1, v_2 \in \mathbb{F}^n, v_1\perp v_2\}.
\end{split}
\end{equation}
\end{theorem}
\begin{proof}
Take first any $T \in \mathcal{ON}_2\cap S^d(\mathbb{F}^n)$ with orthogonal decomposition $T = \sigma_1 \bigotimes_{j=1}^d v_{1j} + \sigma_2 \bigotimes_{j=1}^d v_{2j} $. By possibly permuting the tensor dimensions, we can without loss of generality assume that $v_{11} \perp v_{21}$, so the conclusion follows from Proposition~\ref{prop:symdecomp}. The same argument applies to $\mathcal{SON}_2$ and $\mathcal{CON}_2$.
\end{proof}

\begin{theorem}\label{thm:symrank3}
For any $n$ and any $d>3$, we have
\begin{equation}
\begin{split}
\mathcal{SON}_3 & \cap S^d(\mathbb{F}^n) = \mathcal{CON}_3\cap S^d(\mathbb{F}^n)  \\
&= \Bigl\{\sigma_1 v_1^{\otimes d} + \sigma_2 v_2^{\otimes d} + \sigma_3 v_3^{\otimes d} : \sigma_k \in \mathbb{F}, v_k \in \mathbb{F}^n, 
\langle v_k, v_{k'}\rangle = \delta_{k,k'}\Bigr\},
\end{split}
\end{equation}
and for $d=3$
\begin{equation}
\begin{split}
\mathcal{SON}_3\cap S^3(\mathbb{F}^n) =\Bigl\{\sigma_1 v_1^{\otimes d} + \sigma_2 v_2^{\otimes d} + \sigma_3 v_3^{\otimes d} : \sigma_k \in \mathbb{F}, v_k \in \mathbb{F}^n, 
\langle v_k, v_{k'} \rangle = \delta_{k,k'} \Bigr\}  \\
\cup \Bigl\{\sigma \left( v\otimes w \otimes w + w\otimes v \otimes w + w\otimes w \otimes v \right) :\label{eq:d3son} \sigma \in \mathbb{F}, v, w \in \mathbb{F}^n, v \perp w, \|v\| = \|w\| = 1\Bigr\}.
\end{split}
\end{equation}
\end{theorem}
\begin{proof}
The proof exhaustively considers the possible cases, and we treat the cases $d=3$ and $d > 3$ simultaneously. For any $d \geq 3$, write $T \in \mathcal{SON}_3\cap S^d(\mathbb{F}^n) $ as
\begin{equation}\label{eq:defT}
T = \sigma_1 \bigotimes_{j=1}^d w_{1j} + \sigma_2 \bigotimes_{j=1}^d w_{2j} + \sigma_3 \bigotimes_{j=1}^d w_{3j}.
\end{equation}
If there is some index $j_*$ such that the three vectors $w_{1j_*}, w_{2j_*}, w_{3j_*}$ are mutually orthogonal, then $T\in \mathcal{PCON}_3\cap S^d(\mathbb{F}^n)$, so $ T = \sigma_1 w_{1j_*}^{\otimes d} +\sigma_1 w_{3j_*}^{\otimes d} +\sigma_3 w_{3j_*}^{\otimes d} $ with $\langle w_{kj_*}, w_{k'j_*} \rangle = \delta_{k,k'}$, by Theorem~\ref{prop:symdecomp}.

We proceed by considering the case when there is no such index $j_*$. In the following, we will make repeated use of the fact that strong orthogonality then implies that, for any fixed index $j$, there will be two distinct indices $k$ and $k'$ such that $w_{kj} = w_{k'j}$. Since the term $\bigotimes_{j=1}^d w_{1j} $ is orthogonal to the term $\bigotimes_{j=1}^d w_{3j}$, there will be some index $j$ such that $w_{1j} \perp w_{3j}$. By possibly permuting the dimensions, which does not affect the symmetry of $T$, we can assume that $w_{11} \perp w_{31}$. Because of strong orthogonality, it must be the case that either $w_{21} = w_{11}$ or $w_{21} = w_{31}$ up to complex phase factors. By possibly reordering the first and third terms in the definition of $T$, we can assume that $w_{21} = w_{11}$, after absorbing a phase factor into $\sigma_2$. 

Again, since the term $\bigotimes_{j=1}^d w_{1j} $ is orthogonal to the term $\bigotimes_{j=1}^d w_{2j}$, there will be some index $j$ such that $w_{1j} \perp w_{2j}$. This cannot happen for the first tensor dimension, since $w_{21} = w_{11}$.  By possibly permuting the tensor dimensions, we can assume that this occurs in the second tensor dimension, i.e., $w_{22} \perp w_{12}$. By strong orthogonality, we then have either $w_{32} = w_{12}$ or $w_{32} = w_{22}$ up to complex phase factors. After potentially reordering the first two terms in Equation~\eqref{eq:defT} and absorbing a complex phase factor, we have $w_{32} = w_{22}$. Note that this reordering does not change the assumptions in the first tensor dimension, i.e., that $w_{21} = w_{11} \perp w_{31}$. Summarizing the steps so far, this means that we can write
\begin{equation}
T = \sigma_1 a\otimes \bigotimes_{j=2}^d w_{1j} + \sigma_2 a\otimes u\otimes \bigotimes_{j=3}^d w_{2j} + \sigma_3 w_{31}\otimes u \otimes \bigotimes_{j=3}^d w_{3j},
\end{equation}
for some $a,u$, where $a \perp w_{31}$ and $u\perp w_{12}$. By Lemma~\ref{lemma:symoperations}, it follows that $T\times_1 w_{31} = \sigma_3 u\otimes \bigotimes_{j=3}^d w_{3j}$ is symmetric so $w_{3j} = u$ for $j \geq 2$ after absorbing a phase factor into $\sigma_3$. Likewise, $T\times_2 w_{12} = \sigma_1 a \otimes  \bigotimes_{j=3}^d w_{1j}$ is symmetric so $w_{11} = w_{1j} = a$ for all $j \geq 3$. $T$ is therefore of the form
\begin{equation}\label{eq:defT2}
T = \sigma_1 a \otimes w_{12} \otimes a^{d-2} + \sigma_2 a \otimes u \otimes \bigotimes_{j=3}^d w_{2j} + \sigma_3 w_{31} \otimes u^{\otimes d-1},
\end{equation}
where $a \perp w_{31}$ and $u\perp w_{12}$. In the third tensor dimension, it is by assumption not the case that $u, w_{23}$, and $a$ are all mutually orthogonal. By strong orthogonality, it must then be the case that either $u = w_{23} \perp a$, $u \perp w_{23} = a$, $u = a \perp w_{23}$ or $ u = a = w_{23}$. We study these four cases in turn.

\emph{Case $1$: $u = w_{23} \perp a$}. We have $T\times_3 a = \sigma_1 a\otimes w_{12} \otimes a^{\otimes d-3}$ symmetric, so $w_{12} = a$ after absorbing complex phase factors. This means that $T\times_3 a = \sigma_1 a^{\otimes d-1} = T\times_1 a =  \sigma_1 a^{\otimes d-1} +  \sigma_2 u\otimes u \otimes \bigotimes_{j=4}^d w_{2j}$. The second term in Equation~\eqref{eq:defT2} is then zero, so $T\in \mathcal{SON}_2\cap S^d(\mathbb{F}^n) $ and $T = \sigma_1 v_1^{\otimes d} + \sigma_2 v_2^{\otimes d}$ by Theorem~\ref{thm:symrank2}.

\emph{Case $2$: $u \perp w_{23} = a$}. We have $T \times_3 u = \sigma_3 w_{31} \otimes u^{\otimes d-1} = T \times_2 u = \sigma_3 w_{31} \otimes u^{\otimes d-1} + \sigma_2 a\otimes a \otimes \bigotimes _{j=3}^d w_{2j}$. The second term in Equation~\eqref{eq:defT2} is then zero, so $T\in \mathcal{SON}_2\cap S^d(\mathbb{F}^n) $ and $T = \sigma_1 v_1^{\otimes d} + \sigma_2 v_2^{\otimes d}$ by Theorem~\ref{thm:symrank2}.

\emph{Case $3$: $u = a \perp w_{23}$}. Assume firstly that $w_{12}$ and $w_{31}$ are not parallel. We then claim that the set $w_{12}, a, w_{31}$ is linearly independent. To see this, assume that $0 = \lambda_1 w_{12} + \lambda_2 a + \lambda_3 w_{31}$. Acting on this equation with $a$ implies that $0 = \lambda_2$, so $0 = \lambda_1 w_{12} + \lambda_3 w_{31}$ and $\lambda_1 = \lambda_2 = 0$, since $w_{12}$ and $w_{31}$ are not parallel. By the Hahn-Banach theorem, it then follows that there is some vector $v \in \mathbb{F}^n$ such that $a\cdot v = 0 = w_{31}\cdot v$ and $w_{12}\cdot v \neq 0$. This gives $T \times_1 v = 0 = T \times_2 v = \sigma_1 (w_{12}\cdot v)a^{\otimes d-1}$. The first term in Equation~\eqref{eq:defT2} is then zero, so $T\in \mathcal{SON}_2\cap S^d(\mathbb{F}^n) $ and $T = \sigma_1 v_1^{\otimes d} + \sigma_2 v_2^{\otimes d}$ by Theorem~\ref{thm:symrank2}.

Next, if $w_{12}$ and $w_{31}$ are parallel, then $w_{12} = w_{31}$ after absorbing a complex phase factor. This gives $T \times_1 w_{31} = \sigma_3 a^{\otimes d-1} = T \times_3 w_{31} = \sigma_2 (w_{31}\cdot w_{23}) a\otimes a \otimes \bigotimes_{j=4}^d w_{2j}$. If now $w_{31}\cdot w_{23} = 0$, then it follows that $\sigma_3 a^{\otimes d-1} = 0$, so the third term in Equation~\eqref{eq:defT2} is then zero, $T\in \mathcal{SON}_2\cap S^d(\mathbb{F}^n) $ and $T = \sigma_1 v_1^{\otimes d} + \sigma_2 v_2^{\otimes d}$ by Theorem~\ref{thm:symrank2}. If $w_{31}\cdot w_{23} \neq 0$, then symmetry of $T \times_3 w_{31}$ implies that $w_{2j} = a$ for $j \geq 4$ after absorbing complex phase factors. $T$ is then of the form
\begin{equation}
T = \sigma_1 a\otimes w_{12} \otimes a^{\otimes d-2} + \sigma_2 a\otimes a \otimes w_{23} \otimes a^{\otimes d -3} + \sigma_3 w_{12} \otimes a^{\otimes d-1},
\end{equation}
where $a\perp w_{12}$ and $a\perp w_{23}$. Now, if $w_{12}$ and $w_{23}$ are not parallel, then the set $w_{12}, a, w_{23}$ is linearly independent. To see this, assume that $0 = \lambda_1 w_{12} + \lambda_2 a + \lambda_3 w_{23}$. Acting on this equation with $a$ implies that $0 = \lambda_2$, so $0 = \lambda_1 w_{12} + \lambda_3 w_{23}$ and $\lambda_1 = \lambda_2 = 0$, since $w_{12}$ and $w_{23}$ are not parallel. By the Hahn-Banach theorem, it follows that there is some vector $v \in \mathbb{F}^n$ such that $a\cdot v = 0 = w_{12}\cdot v$ and $w_{23}\cdot v \neq 0$. This gives $T \times_1 v = 0 = T \times_3 v = \sigma a^{\otimes d-1}$, so $T=0$, which is a contradiction.

The only remaining case is $w_{12} = w_{23}$, meaning that $T$ can be written as
\begin{equation}
T = \sigma_1 a\otimes w_{12} \otimes a^{\otimes d-2} + \sigma_2 a\otimes a \otimes w_{12} \otimes a^{\otimes d -3} + \sigma_3 w_{12} \otimes a^{\otimes d-1},
\end{equation}
where $w_{12} \perp a$. For $d = 3$, symmetry implies that $\sigma_1 = \sigma_2 = \sigma_3$, so $T$ is in the second set in Equation~\eqref{eq:d3son}. For $d \geq 4$, we have $T\times_1 w_{12} = \sigma_3 a^{\otimes d-1} = T\times_4 w_{12} = 0$, so $T = 0$, which is a contradiction.
 
 \emph{Case $4$: $u = a = w_{23}$}. We have $T\times_1 w_{31} = \sigma_3 a^{\otimes d-1} = T\times_3 w_{31} = 0$, so $T=0$.
 
 This exhausts all the cases and concludes the proof.
\end{proof}

For tensors $T$ that are not completely orthogonally decomposable, we will provide examples in Section~\ref{sec:examples} that the optimal completely orthogonal rank-$r$ approximation \emph{cannot} in general be computed by successive rank-one deflations, even when explicitly imposing orthogonality constraints. In detail, if we recursively define $T_0 = 0$, $\sigma_i = \langle T, T_i \rangle$ and
\begin{equation}\label{eq:deflateON}
T_{i+1} = v_{i+1,1}\otimes \ldots \otimes v_{i+1, d} = \argmax_{\|y_j\| \leq 1, y_j \perp v_{kj}} \abs{ \langle T - \sum_{k=1}^{i} \sigma_kT_k, y_1 \otimes \ldots \otimes y_d \rangle},
\end{equation}
then we will produce tensors $T$ with
\begin{equation}
\abs{ \langle T, \frac{\sum_{k=1}^rT_k}{\|\sum_{k=1}^rT_k\|} \rangle} < \max_{\substack{Y \in \mathcal{CON}_r \\ \|Y\| \leq 1} }\abs{ \langle T, Y \rangle}.
\end{equation}
Consequently, existence of a symmetric rank decomposition does not follow as in the completely orthogonally decomposable case. In fact, the results in Table~\ref{tab:summary} show that the optimal orthogonal, strongly orthogonal and completely orthogonal rank-$r$ approximations of a symmetric tensor $T$ \emph{cannot} in general be chosen symmetric.

\subsection{Semidefinite programming formulation for symmetric completely orthogonal approximations}\label{sec:SDP}
This section prepares for the proofs of the results in Table~\ref{tab:summary}. The proofs make use of a standard semidefinite programming formulation for the symmetric approximations in Equation~\eqref{eq:max2} in combination with analytical calculations. The semidefinite formulation for the case $r = 1$ is treated by Nie and Wang \cite{nie2014semidefinite} and a comprehensive introduction to polynomial optimization using semidefinite relaxations can be found in a recent monograph \cite{blekherman2012semidefinite}. We distinguish the cases of odd and even dimension $d$.
\subsubsection{Odd dimension $d$}
Let $x = (x_1, \ldots , x_d)$ and define the polynomial
\begin{equation}
p(x) = \sum_{i_1, \ldots , i_d = 1}^n T (i_1, \ldots , i_d) x_{i_1}\ldots x_{i_d}.
\end{equation}
Write $x_k = (x_{k1}, \ldots , x_{kd})$ for each term $k=1, \ldots, r$ in the completely orthogonal approximation. Since the dimension $d$ is odd, we have $p(-x_k) = -p(x_k)$, so we can drop the absolute value signs in Equation~\eqref{eq:max2}, which then is equivalent to the polynomial optimization problem
\begin{equation}\label{eq:oddopt}
  \begin{array}{ll@{}r@{}r@{}l}
    \text{max} &   \displaystyle \sum_{k=1}^r p(x_{k}) \\[\jot]
    \text{s.t.}&   \displaystyle \langle x_{k}, x_{k'}\rangle = 0, \text{ for } k \neq k'  \\
    &      \displaystyle \sum_{k=1}^r \|x_{k}\|^{2d} \leq 1.
  \end{array}
\end{equation}
A standard moment-based relaxation of this problem can be solved using e.g., the existing tools GloptiPoly3 \cite{henrion2009gloptipoly} and YALMIP \cite{Lofberg2004} in MATLAB. A \emph{global} maximum is found by introducing a basis of monomials of the variables $x_{11}, \ldots , x_{1d}, \ldots , x_{rd}$, which is then relaxed to a (convex) semi-definite optimization problem of a specified degree. The global maximizer of the relaxed problem can be found using interior point methods, which guarantees an upper bound to Equation~\eqref{eq:oddopt}. Moreover, the relaxation is guaranteed to be tight provided the relaxation degree is sufficiently large, and typically only a low degree is required. A global maximizer of Equation~\eqref{eq:oddopt} can be automatically extracted from the optimizer of the relaxed problem in GloptiPoly3. The relaxed problem is often of great size, and the large-scale semidefinite solver SDPNAL+ \cite{yang2015sdpnal} was used in our computations.

\subsubsection{Even dimension $d$}
For even $d$, $p(x) = p(-x)$ and the optimal completely orthogonal rank-$r$ approximation of $T$ in Equation~\eqref{eq:max2} equals $\sum_{k=1}^r t_kp(x^{(k)})$, where $t_k \in \{ -1 , 1\}$. We then consider the polynomial optimization problem
\begin{equation}\label{eq:even1}
  \begin{array}{ll@{}r@{}r@{}l}
    \text{max} & \displaystyle \sum_{k=1}^r t_kp(x_{k}) \\[\jot]
    \text{s.t.}& \displaystyle \langle x_{k}, x_{k'}\rangle = 0, \text{ for } k \neq k'  \\
    &    \displaystyle \sum_{k=1}^r \|x_{k}\|^{2d} \leq 1 \\
    & \displaystyle -1 \leq t_k \leq 1.
  \end{array}
\end{equation}
The optimal solution clearly has $t_k \in \{-1, +1\}$. Just as in the case of odd dimension, these problems can be solved using a moment-based relaxation.

\subsection{Main results}\label{sec:main_proofs}
This section contains the proofs of the results in Table~\ref{tab:summary}. We provide examples of where none of the optimal orthogonal rank-$r$ approximations of a symmetric tensor $T$ can be chosen symmetric. In the case of completely orthogonal approximations, we also give some stronger conditions which do result in the existence of symmetric optimizers.

\subsubsection{$\mathcal{CON}_r$ and $\mathcal{PCON}_{r,P}$ approximations}
We first consider the cases $\mathcal{CON}_r$ and $\mathcal{PCON}_{r,P}$, and first show that the optimal rank-$r$ approximation to a symmetric tensor \emph{cannot} in general be chosen symmetric. In fact, we show that this occurs for a wide class of tensors.

\begin{theorem}\label{thm:main}
Let $x_1, x_2, x_3$ be three mutually orthonormal vectors in $\mathbb{R}^3$. Define the symmetric tensor $T\in S^3(\mathbb{R}^3)$ by
\begin{equation}
T = \frac{1}{6}\sum_{\sigma \in S^3} x_{\sigma(1)}\otimes x_{\sigma(2)} \otimes x_{\sigma(3)}.
\end{equation}
There is then no optimal rank-$3$ completely orthogonal approximation of $T$ that is symmetric.
\end{theorem}
\begin{proof}
Since $x_1,x_2,x_3$ are mutually orthogonal, we can without loss of generality perform a change of basis to assume $x_1 = e_1, x_2 = e_2, x_3 = e_3$. When the same change of basis is applied to a completely orthogonal approximation of $T$, this does not change the fact that the approximation is completely orthogonal. The optimal symmetric and completely orthogonal approximation $Y_{s}$ of $T = \frac{1}{6}\sum_{\sigma \in S^3} e_{\sigma(1)}\otimes e_{\sigma(2)} \otimes e_{\sigma(3)}
$ can be found by the procedure in Sec.~\ref{sec:SDP}. This results in
\begin{equation}
Y_s = \frac{4}{27} \left( y_1^{\otimes 3} + y_2^{\otimes 3} + y_3^{\otimes 3}\right),
\end{equation}
where $y_1 = \frac{1}{3}[-2, -2, 1]^T$, $y_1 = \frac{1}{3}[-2, 1, -2]^T$, $y_1 = \frac{1}{3}[1, -2, -2]^T$ to within machine precision. The resulting approximation error is
\begin{equation}
\frac{\|T-Y_{s}\|}{\|T\|} = 0.7778.
\end{equation}
However, the completely orthogonal (but not symmetric) tensor $Y_{ns}$ defined by
\begin{equation}
Y_{ns} = \frac{1}{6}\left( e_1\otimes e_2 \otimes e_3 + e_2\otimes e_3 \otimes e_1  + e_3\otimes e_1 \otimes e_2 \right)
\end{equation}
has approximation error
\begin{equation}
\frac{\|T-Y_{ns}\|}{\|T\|} =  0.7071,
\end{equation}
which concludes the proof.
\end{proof}

In contrast to the above result, we next prove that in the case $n = 2$, the optimal completely orthogonal approximation can always be taken symmetric. For $n=2$, $r$ is either $1$ or $2$. The case $r=1$ is exactly Theorem~\ref{thm:banach}, and we next show the case $r=2$.

\begin{theorem}\label{thm:mainn2}
If $T \in S^d(\mathbb{F}^{2})$ is symmetric and $n = r = 2$, then the optimal completely orthogonal rank-$2$ approximation of $T$ can be chosen symmetric. i.e.
\begin{equation}
\begin{split}
\max_Y &\left\{ \abs{ \langle T,  Y \rangle }:  Y \in \mathcal{CON}_2,  \|Y\| \leq 1  \right\}  \\
&= \max_{v,w, \sigma_k} \left\{ \abs{ \langle T,  \sigma_1 v^{\otimes d} + \sigma_2 w^{\otimes d} \rangle }: v\perp w, \|v\| = \|w\| = 1,  \sigma_1^2 + \sigma_2^2 = 1 \right\}.
\end{split}
\end{equation}
\end{theorem}
\begin{proof}
Let $R = \bigl[ \begin{smallmatrix} 0 &-1\\1 & 0  \end{smallmatrix} \bigr]$. For $r = n = 2$, Equation~\eqref{eq:max2} is equivalent to the maximization problem
\begin{equation}\label{eq:n2max}
\begin{split}
\max_{\mu_k, v_j} \left\{ \abs{ \mu_1 \langle T, v_1\otimes \ldots \otimes v_d \rangle +  \mu_2 \langle T, Rv_1\otimes \ldots \otimes Rv_d \rangle }: \mu_1^2 + \mu_2^2 \leq 1, \|v_k\| \!\!= \!\! 1\right\} \\
= \max_{\mu_k, v_j} \left\{ \abs{  \langle \mu_1 T + \mu_2 T\times_1 R^T \times_2 \ldots \times_d R^T \!\!, v_1\otimes \ldots \otimes v_d \rangle  }: \mu_1^2 + \mu_2^2 \leq 1, \|v_k\| = 1\right\}. 
\end{split}
\end{equation}
For any $\mu_1, \mu_2$, the tensor $\mu_1 T + \mu_2 T\times_1 R^T \times_2 \ldots \times_d R^T$ is a sum of two symmetric tensors, and hence symmetric. By Theorem~\ref{thm:banach}, the maximizer of Equation~\eqref{eq:n2max} can be chosen symmetric i.e., $v_1 \otimes \ldots \otimes v_d = v^{\otimes d}$. This implies that both terms in the completely orthogonal approximation are symmetric.
\end{proof}

Next, we extend this theorem also to partially orthogonal approximations. Note that if $T\in S^d(\mathbb{F}^n)$ is symmetric and $Y \in \mathcal{PCON}_{r,P}$, then $\langle T, Y \rangle = \langle T, Y_\sigma \rangle$, where $Y_{\sigma} (i_1, \ldots , i_d) = Y (i_{\sigma(1)}, \ldots , i_{\sigma(d)})$. It follows that approximations in $\mathcal{PCON}_{r,P}$ and $\mathcal{PCON}_{r,\{1, \ldots , \abs{P}\}}$ result in the same approximation errors to symmetric tensors. We will therefore in this section identify $\mathcal{PCON}_{r,P}$ and $\mathcal{PCON}_{r,\{1, \ldots , \abs{P}\}}$ without further comment.

\begin{theorem}\label{thm:mainn2partial}
If $T \in S^d(\mathbb{F}^{2})$ is symmetric, $n = r = 2$, and $P \subseteq \{1, \ldots , d\}$, then the optimal completely orthogonal rank-$2$ approximation of $T$ has symmetric terms in each of the two disjoint index sets $P$ and $\{1, \ldots , d\} \smallsetminus P$, i.e.
\begin{equation}
\begin{split}
\max_Y &\left\{ \abs{ \langle T,  Y \rangle }: Y \in \mathcal{PCON}_{2,P},  \|Y\| \leq 1  \right\}  \\
&= \max_{v_k, \sigma_k} \left\{ \abs{ \langle T,  \sum_{k=1}^2\sigma_k w_{k}^{\otimes \abs{P}} \otimes v_{k}^{\otimes d - \abs{P}} \rangle }: \langle w_{k}, w_{k'} \rangle = \delta_{k,k'}, \sum_{k=1}^2 \sigma_k^2 = 1 \right\},
\end{split}
\end{equation}

\end{theorem}

\begin{proof}
Let $R = \bigl[ \begin{smallmatrix} 0 &-1\\1 & 0  \end{smallmatrix} \bigr]$, and let $Y = \sum_{k=1}^2 \sigma_k \bigotimes_{j=1}^d w_{kj}$ be a maximizer of Equation~\eqref{eq:max2}. We first fix the terms $\bigotimes_{j=\abs{P}+1}^d w_{kj}$, for $k=1,2$. For $r = n = 2$, the terms $\bigotimes_{j=1}^{\abs{P}} w_{kj}$ are maximizers of the expression
\begin{equation}\label{eq:n2maxpartial}
\begin{split}
& \max_{\|v_j\|=1} \big\{ \abs{ \sigma_1 \langle T, v_1\otimes \ldots \otimes v_{\abs{P}} \otimes w_{1, \abs{P}+1} \otimes \ldots \otimes w_{1d} \rangle \\
&\quad +  \sigma_2 \langle T, Rv_1\otimes \ldots \otimes Rv_{\abs{P}} \otimes w_{1, \abs{P}+1} \otimes \ldots \otimes w_{1d} \rangle }\big\} \\
&= \max_{\|v_j\|=1} \big\{ \abs{  \langle \sigma_1 T\times_{\abs{P}+1}w_{1, \abs{P}+1} \ldots \times_{d}w_{1d} \\
&\quad + \sigma_2 T\times_1 R^T \times_2 \ldots \times_{\abs{P}} R^T \times_{\abs{P}+1}w_{1, \abs{P}+1} \ldots \times_{d}w_{1d}, v_1\otimes \ldots \otimes v_d \rangle  }\big\}. 
\end{split}
\end{equation}
For any $\sigma_1, \sigma_2$ and vectors $w_{1, \abs{P}+1} \ldots \times_{d}w_{1d}$, the tensor $\sigma_1 T\times_{\abs{P}+1}w_{1, \abs{P}+1} \ldots \times_{d}w_{1d} + \sigma_2 T\times_1 R^T \times_2 \ldots \times_{\abs{P}} R^T \times_{\abs{P}+1}w_{2, \abs{P}+1} \ldots \times_{d}w_{2d} $ is a sum of two symmetric tensors, and hence symmetric. By Theorem~\ref{thm:banach}, the maximizer of Equation~\eqref{eq:n2maxpartial} can be chosen symmetric i.e., we can replace $\bigotimes_{j=1}^{\abs{P}} w_{1j}$ by a symmetric maximizer $v_1 \otimes \ldots \otimes v_d = v^{\otimes d}$, with also $\bigotimes_{j=1}^{\abs{P}} w_{2j}$ replaced by the symmetric maximizer $(Rv)^{\otimes d}$.

Fixing the terms $v^{\otimes d}$ and $(Rv)^{\otimes d}$, the first term $\bigotimes_{j=\abs{P}+1}^d w_{1j}$ is a maximizer of
\begin{equation}\label{eq:n2maxpartial2}
\begin{split}
\max_{u_j} \big\{ \abs{ \sigma_1 \langle T, v\otimes \ldots \otimes v \otimes u_{\abs{P}+1} \otimes \ldots \otimes u_{d} \rangle }: \|u_j\| \!\!= \!\! 1\big\} \\
= \max_{ u_j} \big\{ \abs{  \langle \sigma_1 T\times_{1}v \ldots \times_{\abs{P}}v , u_{\abs{P}}\otimes \ldots \otimes u_d \rangle  }: \|u_j\| = 1\big\}. 
\end{split}
\end{equation}

Since the tensor $\sigma_1 T\times_{1}v \ldots \times_{\abs{P}}v$ is symmetric, $\bigotimes_{j=\abs{P}+1}^d w_{kj}$ can be replaced by a symmetric maximizer $w_1^{\otimes d}$. In the same way, the second term $\bigotimes_{j=\abs{P}+1}^d w_{2j}$ can be replaced by a symmetric term $w_2^{\otimes d}$.
\end{proof}

In the same way as the second part of Theorem~\ref{thm:mainn2partial}, we can show the following.
\begin{theorem}\label{theorem:mainpartial}
If $T \in S^d(\mathbb{F}^{n})$ is symmetric, and $P \subseteq \{1, \ldots , d\}$, then the optimal partially orthogonal rank-$r$ approximation of $T$ can be partitioned into two disjoint index sets, one of which has symmetric terms, i.e.
\begin{equation}
\begin{split}
\max_Y &\left\{ \abs{ \langle T,  Y \rangle }: Y \in \mathcal{PCON}_{r,P},  \|Y\| \leq 1  \right\}  \\
&= \max_{v_k, \sigma_k} \left\{ \abs{ \langle T,  \sum_{k=1}^r\sigma_k \bigotimes_{j=1}^{\abs{P}}w_{kj}\otimes v_{k}^{\otimes d - \abs{P}} \rangle }: \langle w_{k}, w_{k'} \rangle = \delta_{k,k'}, \sum_{k=1}^r \sigma_k^2 = 1 \right\}.
\end{split}
\end{equation}
\end{theorem}

We lastly show that imposing a stronger condition than complete orthogonality results in the existence of symmetric optimizers. Given $r$ tensors $v_{11}\otimes \ldots \otimes v_{1d}, \ldots , v_{r1}\otimes \ldots \otimes v_{rd}$, we study approximations $v_{kj}$ under the condition that
\begin{equation}\label{eq:strong_assumption}
\langle v_{kj} , v_{k'j'} \rangle = 0, \quad \text{ for all } k \neq k' \text{ and all } 1\leq j,j' \leq d.
\end{equation}
Note that this condition does not vacuously imply $v_{kj} = v_{kj'}$ for all $k,j,j'$ if $rd \leq n$. However, we now show that the optimal approximation to a symmetric tensor $T$ can be chosen to have $v_{kj} = v_{kj'}$.
\begin{theorem}\label{thm:mainentirely}
If $T \in S^d(\mathbb{F}^{n})$ is symmetric and $r \leq n$, then the optimal rank-$r$ approximation $\sigma_1 v_{11}\otimes \ldots \otimes v_{1d} + \ldots + \sigma_r v_{r1}\otimes \ldots \otimes v_{rd}$ obeying the condition in Equation~\eqref{eq:strong_assumption} can be chosen symmetric, i.e., of the form $\sigma_1 v_1^{\otimes d} + \ldots + \sigma_r v_r^{\otimes d}$.
\end{theorem}

\begin{proof}
Fix a maximizer $\{w_{kj}\}$ of Equation~\eqref{eq:max2} obeying Equation~\eqref{eq:strong_assumption}. Write $V_k = \text{span}\left(w_{k1}, \ldots , w_{kd} \right)$. It follows that $V_k \perp V_{k'}$ for any $k\neq k'$. Write $P_k$ as the matrix for the projection onto $V_k$. Consider the maximization problem
\begin{equation}\label{eq:maxentire}
\max_{v_{kj}} \left\{ \sum_{k=1}^r \abs{ \langle T\times_1 P_k^T \times_2 \ldots \times_d P_k^T, \bigotimes_{j=1}^d v_{kj} \rangle }^2 : \|v_{kj}\| = 1 \right\}
\end{equation}
where no orthogonality conditions are explicitly imposed. Given any maximizer $\{v_{kj}\}$ of Equation~\eqref{eq:maxentire}, it is clear that the tensors $P_1v_{11}\otimes \ldots \otimes P_1v_{1d}, \ldots , P_r v_{r1}\otimes  \ldots \otimes P_rv_{rd}$ satisfy Equation~\eqref{eq:strong_assumption} with $\sum_{k=1}^r \abs{ \langle T\times_1 P_k^T \times_2 \ldots \times_d P_k^T, \bigotimes_{j=1}^d v_{kj} \rangle}^2 = \sum_{k=1}^r \abs{\langle T \bigotimes_{j=1}^d P_kv_{kj} \rangle}^2$. The maximum of Equation~\eqref{eq:maxentire} is therefore no greater than the maximum of Equation~\eqref{eq:max1}. Plugging in the tensors $\{w_{kj}\}$ into Equation~\eqref{eq:maxentire} shows that these maxima in fact coincide. Since each tensor $T\times_1 P_k^T \times_2 \ldots \times_d P_k^T$ is symmetric, Theorem~\ref{thm:banach} now shows that for any $k$, the term $\bigotimes_{j=1}^d v_{kj} $ in Equation~\eqref{eq:maxentire} can be chosen symmetric, which concludes the proof.
\end{proof}

\subsubsection{$\mathcal{ON}_r$ and $\mathcal{SON}_r$ approximations}\label{sec:counterON}
This section provides examples that show that optimal orthogonal and strongly orthogonal approximations of symmetric tensors cannot be chosen symmetric, in general. Not only does there not exist a decomposition of the optimal approximation with each term symmetric, but the approximation as a whole can in general not be taken symmetric (cf. the discussion following Theorem~\ref{thm:banach}). We provide two examples: one with $r \leq n$ and one with $ r > n$.

\begin{theorem}\label{thm:no_SON}
The tensor $T \in S^3(\mathbb{F}^n)$ defined by
\begin{equation}
T = \left(
\begin{array}{cc| c c} 
0 & 1 & 1 & 0 \\ 1 & 0 & 0 & 2
\end{array}
\right)
\end{equation}
has no optimal strongly orthogonal rank-$3$ approximation that is symmetric.
\end{theorem}

\begin{proof}
We will produce a strongly orthogonal tensor with lower approximation error than the optimal symmetric strongly orthogonal approximation. We first consider symmetric strongly orthogonal approximations. Since no three non-zero vectors in $\mathbb{F}^2$ can be mutually orthogonal, any $S \in \mathcal{SON}_3\cap S^3(\mathbb{F}^2) $ is either of the form (i) $\sigma \left( w\otimes v \otimes v + v \otimes w \otimes v + v\otimes v \otimes w \right)$ or (ii) $\sigma_1 v^{\otimes 2} + \sigma_2 w^{\otimes 2}$ for $v \perp w$, by Theorem~\ref{thm:symrank3}. For the form (i), we write $ v = \left[ v_1, v_2 \right]^T$, and can then put $ w = e^{i\phi}\left[  -v_2, v_1 \right]^T$ for some phase $\phi$. We have
\begin{align}
\abs{\langle T, w\otimes v \otimes v \rangle}^2 = \abs{(v_1^2 - v_2^2)v_1 + (-v_2v_1 + 2v_2v_1)v_2}^2 = \abs{v_1}^6,
\end{align}
with maximum achieved for $v = e^{i\varphi} e_1$ for some phase $\varphi$. The remaining terms $\abs{\langle T, v\otimes w \otimes v \rangle}^2, \abs{\langle T, v\otimes v \otimes w \rangle}^2$ also achieve their maxima for this $v$, so the approximation is $S = \langle T, w\otimes v \otimes v \rangle w\otimes v \otimes v + \langle T, v\otimes w \otimes v \rangle v\otimes w \otimes v + 
\langle T, v\otimes v \otimes w \rangle v\otimes v \otimes w  = e_2 \otimes e_1 \otimes e_1 +  e_1 \otimes e_2 \otimes e_1 + e_1 \otimes e_1 \otimes e_2$ with approximation error $\|T-S\| = 2$.

For the form (ii), if $w \neq 0$, we can write $v = \left[ v_1, v_2 \right]^T$, $ w = e^{i\phi}\left[  -v_2, v_1 \right]^T$ and obtain
\begin{equation}
\begin{split}
\abs{\langle T, v^{\otimes 3}\rangle}^2 + \abs{\langle T, w^{\otimes 3}\rangle}^2 = \left(\abs{v_2}\abs{3v_1^2 + 2v_2^2}\right)^2 +  \left(\abs{v_1}\abs{3v_2^2 + 2v_1^2}\right)^2 \\
\leq \abs{v_2}^2\left( 3\abs{v_1}^2 + 2\abs{v_2}^2 \right)^2 + \abs{v_1}^2\left( 3\abs{v_2}^2 + 2\abs{v_1}^2 \right)^2
\end{split}
\end{equation}
with equality precisely when $v_1$ and $v_2$ have the same complex phase. This is maximized for $\abs{v_1} = \frac{1}{\sqrt{2}}$, and the resulting approximation
\begin{equation}
S_1 = \left(
\begin{array}{cc| c c} 
0 & \frac{5}{4} & \frac{5}{4} & 0 \\ \frac{5}{4} & 0 & 0 & \frac{5}{4}
\end{array}
\right)
\end{equation}
has approximation error $\|T-S_1\| = \frac{\sqrt{3}}{2}$. If $w = 0$, the maximization of $\abs{\langle T, v^{\otimes 3}\rangle}^2 = \left(\abs{v_2}\abs{3v_1^2 + 2v_2^2}\right)^2$ gives $v = e_2$ with resulting approximation
\begin{equation}
S_2 = \left(
\begin{array}{cc| c c} 
0 & 0 & 0 & 0 \\ 0 & 0 & 0 & 2
\end{array}
\right)
\end{equation}
and approximation error $\| T-S_2\| = \sqrt{3}$. However, defining $a = \frac{1}{\sqrt{2}}[1 ,1]^T$, $b = [0.8321 , 0.5547]^T$, $c = \frac{1}{\sqrt{2}}[1 ,1]^T$ and $a_{\perp} = [-a_2, a_1]^T$ and likewise for $b_{\perp}, c_{\perp}$, the strongly orthogonal approximation $S_3 = \langle T, a\otimes b\otimes c\rangle a\otimes b\otimes c + \langle T, a_{\perp}\otimes b_{\perp}\otimes c_{\perp} \rangle a_{\perp}\otimes b_{\perp}\otimes c_{\perp}+ \langle T, a\otimes b_{\perp}\otimes c \rangle a\otimes b_{\perp}\otimes c $ evaluates numerically to
\begin{equation}
S_3 = \left(
\begin{array}{cc| c c} 
0 & 1.5001 & 1.0001 & 0 \\ 1.0001 & 0 & 0 &1.5001
\end{array}
\right)
\end{equation}
with approximation error $ \|T-S_3\|  \approx 0.7071 < \min(\frac{\sqrt{3}}{2}, \sqrt{3}, 2)$. Since $S_3$ is not symmetric, this concludes the proof.
\end{proof}

\begin{theorem}\label{thm:no_ON}
The symmetric tensor $T \in S^4(\mathbb{R}^2)$ defined by
\begin{equation}
T = e_1 \otimes e_1 \otimes e_1 \otimes e_2 +  e_1 \otimes e_1 \otimes e_2 \otimes e_1 +  e_1 \otimes e_2 \otimes e_1 \otimes e_1 +  e_2 \otimes e_1 \otimes e_1 \otimes e_1
\end{equation}
has no optimal orthogonal, strongly orthogonal or partially orthogonal rank-$2$ approximation that is symmetric.
\end{theorem}
\begin{proof}
Because of Theorem~\ref{thm:symrank2}, it will be enough to produce a strongly orthogonal tensor of rank no greater than $2$ with strictly lower approximation error than the optimal completely orthogonal rank-$2$ approximation. The procedure in Section~\ref{sec:SDP} gives the optimal completely orthogonal rank-$2$ approximation
\begin{equation}
Y_s = -y_1^{\otimes 4} + y_2^{\otimes 4},
\end{equation}
where $y_1 = \frac{1}{\sqrt{2}} [1,-1]^T$, $y_2 = \frac{1}{\sqrt{2}} [-1,-1]^T$ with approximation error $\|T - Y_s\| = \sqrt{2}$. However, the strongly orthogonal tensor
\begin{equation}
Y = \frac{3}{\sqrt{8}}e_1\otimes u^{\otimes 3} + \frac{3}{\sqrt{8}}e_1\otimes v^{\otimes 3},
\end{equation}
with $u = \frac{1}{\sqrt{2}}\bigl[ -1, 1 \bigr]^T, v = \frac{1}{\sqrt{2}} \bigl[  1,  1 \bigr]^T $ has approximation error $\|T - Y\| =\sqrt{\frac{7}{4}} < \sqrt{2}$. Since $Y$ is not symmetric, this concludes the proof.
\end{proof}
\subsection{Additional differences from the matrix case}\label{sec:examples}
This section gives three additional examples of properties of orthogonal approximations that differ starkly between the matrix case and the general tensor case.

As is well known, the Eckart-Young theorem states that the optimal rank-$r$ (orthogonal) approximation of a matrix $A$ is given by $\sum_{k=1}^r \sigma_k u_kv_k^T$, where $\sigma_1 \geq \sigma_2 \geq \ldots \geq 0$ are the singular vectors of $A$ and $u_k, v_k$ are the left and right singular vectors, respectively. By the min-max theorem, these can be computed by successive rank-one deflations as in Equation~\eqref{eq:deflateON}. Theorem~\ref{thm:main} implies that this is not true in general, for $n \geq 3$, since rank-$1$ approximations can be chosen symmetric, meaning that the successive deflations can all be taken symmetric. Since the optimal orthogonal approximation is not always symmetric, they cannot coincide, in general. The following example shows the same statement for $n = 2$.

\begin{example}
Let $T \in S^3({\mathbb{R}^2})$ be the tensor defined by
\begin{equation}\label{eq:Tex}
T = \left(
\begin{array}{cc| c c} 
0 & 1 &1 & 0 \\ 1 & 0 & 0 &2
\end{array}
\right)
\end{equation}
Since $n=2$, the optimal completely orthogonal approximation can be taken symmetric by Theorem~\eqref{thm:mainn2}. We compare the optimal orthogonal rank-$2$ approximation to the result of the successive orthogonal deflation in Equation~\eqref{eq:deflateON}. As shown in Theorem~\ref{thm:no_SON}, the optimal rank-$2$ completely orthogonal approximation is given by 
\begin{equation}
S_1 = \left(
\begin{array}{cc| c c} 
0 & \frac{5}{4} &\frac{5}{4} & 0 \\ \frac{5}{4}  & 0 & 0 &\frac{5}{4}
\end{array}
\right),
\end{equation}
and the optimal rank-$1$ approximation is
\begin{equation}
S_2= \left(
\begin{array}{cc| c c} 
0 & 0 & 0 & 0 \\ 0  & 0 & 0 & 2
\end{array}
\right).
\end{equation}
Performing a successive orthogonal deflation as in Equation~\eqref{eq:deflateON} results in the second term being the zero tensor. Since $\| T - S_1\| < \|T  - S_2\|$, the two-fold orthogonal deflation gives larger approximation error than the optimal rank-$2$ completely orthogonal approximation.
\end{example}

A vector $x\in \mathbb{F}^n$ is a singular vector \cite{lim2006singular,qi2005eigenvalues} of the symmetric tensor $T\in\mathbb{F}^{n\times \ldots \times n}$ with singular value $\sigma$ if $T\times_2 x \times \ldots \times_d x = \sigma x$. For $r=1$, the solution $x$ to Equation~\eqref{eq:max1} is known to be a singular vector with $T\times_2 x \times \ldots \times_d x = \langle T, x\otimes \ldots \otimes x\rangle x$. This is not the case for $r \geq 2$, as we now show.
\begin{example}Let $T$ be as in Equation~\eqref{eq:Tex}. As shown in Theorem~\ref{thm:no_SON}, the optimal rank-$2$ completely orthogonal approximation is given by the terms $\frac{5}{2\sqrt{2}} v\otimes v \otimes v + \frac{5}{2\sqrt{2}} w\otimes w \otimes w$, where $v = \frac{1}{\sqrt{2}} [1, 1]^T, w = \frac{1}{\sqrt{2}} [-1, 1]^T$, so $T\times_2 v \times \ldots \times_d v = [1, \frac{3}{2}]^T$, which is not a multiple of $v$.
\end{example}

Lastly, we present an example where the optimal completely orthogonal approximations coincide for $r=2$ and $r=3$, but without being equal to the approximated tensor $T$. This situation is unique to tensors with dimension $d \geq 2$ and does not occur in the matrix case $d=2$.
\begin{example}
We consider the tensor $T\in S^3(\mathbb{R}^3)$ defined by $T = e_1\otimes e_1 \otimes e_2 + e_1\otimes e_2 \otimes e_1 + e_2\otimes e_1 \otimes e_1$. The procedure in Section~\ref{sec:SDP} results in an optimal rank-$2$ completely orthogonal approximation $Y$ with terms $ v\otimes v \otimes v$ and $w\otimes w \otimes w$, where $v =  \frac{1}{\sqrt{2}} [1,1,0]^T$ and $w = \frac{1}{\sqrt{2}} [-1,1,0]$, respectively, to machine precision. Using the same procedure shows that this coincides with the optimal rank-$3$ approximation, up to machine precision, and the approximation error is $\frac{\sqrt{3}}{2}$, i.e., non-zero.
\end{example}

\subsection{NP-hardness of optimal rank-$r$ (strongly, completely, partially) orthogonal approximation}\label{sec:NP}
For tensors of dimension $d > 2$, it is well-known that finding the optimal rank-one approximation over $\mathbb{F}$ is NP-hard in general for $\mathbb{F} = \mathbb{R}$ and $d\geq 3$ \cite{hillar2013most}, or $\mathbb{F} = \mathbb{C}$ and $d \geq 4$  \cite{friedland2018nuclear}. However, this does not immediately translate into the corresponding result for $r > 1$, since we showed in Section~\ref{sec:examples} that the optimal rank-$r$ approximation is in general unrelated to the optimal rank-one approximation. This section therefore constructs a straight-forward polynomial-time reduction from finding the optimal rank-$r$ approximation to finding the optimal rank-one approximation, which shows that the situation is NP-hard for any $r \geq 1$ and $d\geq 3$ for $\mathbb{F} = \mathbb{R}$, or $d \geq 4$ for $\mathbb{F} = \mathbb{C}$.

Given an integer $r$ and a vector $v\in \mathbb{F}^n$, define $\mathcal{B}_k(v) \in  \mathbb{F}^{rn }$ by
\begin{equation}
\mathcal{B}_\ell(v)(i) = \begin{cases}
v(i), & (\ell-1)n < i \leq \ell n, \\
0, & \text{otherwise}.
\end{cases}
\end{equation}
For a tensor $T = \sum_{k=1}^r \sigma_k \bigotimes_{j=1}^d v_{kj}$, define $\mathcal{B}_\ell(T) = \sum_{k=1}^{r} \sigma_k \bigotimes_{j=1}^d \mathcal{B}_\ell(v_{kj})$. A block-diagonal tensor with copies of $T$ on the diagonal is then given by $\mathcal{B}_1(T) + \ldots + \mathcal{B}_r(T)$. The following is a straightforward calculation.

\begin{proposition}\label{prop:blockr}
An optimal (strongly, completely, partially) orthogonal rank-$r$ decomposition of $\mathcal{B}_1(T) + \ldots + \mathcal{B}_r(T)$ is $\mathcal{B}_1(u) + \ldots + \mathcal{B}_r(u)$, where $u$ is the optimal rank-$1$ approximation of $T$.
\end{proposition}
\begin{proof}
Let $S = \sum_{k=1}^r \sigma_k \bigotimes_{j=1}^d v_{kj}$ be any sum of (strongly, completely, partially) orthogonal tensors with $v_{kj} \in \mathbb{F}^{rn_j}$ and $\|v_{kj}\| = 1$. Write $v_{kj}^T = \left[ v_{kj, (1)}^T \ldots v_{kj,(r)}^T\right]$ for $v_{kj,(\ell)} \in \mathbb{F}^{n_j}$. Then for fixed $k$
\begin{equation}\label{eq:blockmax}
\abs{ \langle \mathcal{B}_1(T) + \ldots + \mathcal{B}_r(T),  \bigotimes_{j=1}^d v_{kj}\rangle }= \abs{ \sum_{\ell=1}^r \langle T,  \bigotimes_{j=1}^d v_{kj, (\ell)}\rangle }\leq  \abs{ \langle T,  u \rangle} \sum_{\ell=1}^r  \prod_{j=1}^d \|v_{kj, (\ell)}\|.
\end{equation}
By the AM-GM inequality, we see that
\begin{equation}
 \sum_{\ell=1}^r  \prod_{j=1}^d \|v_{kj, (\ell)}\| \leq \sum_{\ell = 1}^r \frac{1}{d} \sum_{j=1}^d  \|v_{kj, (\ell)}\|^2 = \frac{1}{d} \sum_{j=1}^d \sum_{\ell = 1}^r   \|v_{kj, (\ell)}\|^2 = 1,
\end{equation}
since $\sum_{\ell=1}^r \|v_{kj, (\ell)}\|^{2} = \|v_{kj}\|^{2} = 1$, for every $j$. Inserting this into Equation~\eqref{eq:blockmax} shows that
\begin{equation}
\sum_{k=1}^r \abs{ \langle T^{(r)},  \bigotimes_{j=1}^d v_{kj}\rangle}^2 \leq r \abs{\langle T,  u \rangle}^2,
\end{equation}
and this equality is clearly achieved with the choice $\bigotimes_{j=1}^d v_{kj} = \mathcal{B}_{k}(u)$ for every $k$. Since this results in a completely orthogonal decomposition, this concludes the proof.
\end{proof}

\begin{theorem}
For $T\in \mathbb{F}^{n_1\times \ldots \times n_d}$ and any $r \geq 1$, It is in general NP-hard to approximate any of $\|T\|_{\mathcal{ON}_r, \mathbb{F}}$, $\|T\|_{\mathcal{SON}_r, \mathbb{F}}$, $\|T\|_{\mathcal{CON}_r, \mathbb{F}}$, $\|T\|_{\mathcal{PCON}_{r,P}, \mathbb{F}}$ to arbitrary accuracy for $d\geq 3$ if $\mathbb{F} = \mathbb{R}$, and $d \geq 4$ if $\mathbb{F} = \mathbb{C}$. For real-valued symmetric tensors $T \in S^d(\mathbb{R}^n)$, computing any of $\|T\|_{\mathcal{ON}_r, \mathbb{F}}$, $\|T\|_{\mathcal{SON}_r, \mathbb{F}}$, $\|T\|_{\mathcal{CON}_r, \mathbb{F}}$, $\|T\|_{\mathcal{PCON}_{r,P}, \mathbb{F}}$ is also NP-hard.
\end{theorem}
\begin{proof}
Since $r \leq \min_j(n_j)$, $\mathcal{B}_1(T) + \ldots + \mathcal{B}_r(T)$ has dimensions bounded by $n^2$, so Proposition~\ref{prop:blockr} gives a polynomial reduction from the problem of approximating $\|\mathcal{B}_1(T) + \ldots + \mathcal{B}_r(T)\|_{\mathcal{ON}_r, \mathbb{F}}$ to approximating $\|T\|_{\sigma, \mathbb{F}}$, which is NP-hard in general (\cite[Theorem~1.11]{hillar2013most} for $\mathbb{R}$ and \cite[Corollary~8.7]{friedland2018nuclear} for $\mathbb{C}$). This shows the first statement. For the second, note that $\mathcal{B}_1(T) + \ldots + \mathcal{B}_r(T)$ is symmetric when $T$ is symmetric and apply Theorem~10.2 in \cite{hillar2013most}.
\end{proof}

\section{Conclusions}
We have shown that optimal orthogonal approximations to symmetric tensors exhibit significant differences to their counterparts for symmetric matrices. Under any of the notions of tensor orthogonality, we have provided examples where there are no optimal approximations that are symmetric. This is an analogue of Comon's conjecture in the setting of orthogonal approximation, and is different from the matrix case and the rank-$1$ approximation for general tensors. Moreover, we have given examples where the optimal approximations cannot be calculated by an iterative deflation of the optimal rank-$1$ approximations and where the approximating terms are not tensor singular values. We have also shown some structural results on symmetric orthogonally decomposable tensors that might be of independent interest.

\bibliographystyle{amsplain}
\bibliography{references}

\providecommand{\bysame}{\leavevmode\hbox to3em{\hrulefill}\thinspace}
\providecommand{\MR}{\relax\ifhmode\unskip\space\fi MR }
\providecommand{\MRhref}[2]{%
  \href{http://www.ams.org/mathscinet-getitem?mr=#1}{#2}
}
\providecommand{\href}[2]{#2}
\begin{thebibliography}{10}

\bibitem{anandkumar2014tensor}
Animashree Anandkumar, Rong Ge, Daniel Hsu, Sham~M Kakade, and Matus Telgarsky,
  \emph{{Tensor decompositions for learning latent variable models}}, J. Mach.
  Learn. Res. \textbf{15} (2014), 2773--2832.

\bibitem{banach1938homogene}
Stefan Banach, \emph{{{\"U}ber homogene Polynome in ($ L^2$)}}, Studia
  Mathematica \textbf{7} (1938), no.~1, 36--44.

\bibitem{blekherman2012semidefinite}
Grigoriy Blekherman, Pablo~A Parrilo, and Rekha~R Thomas, \emph{Semidefinite
  optimization and convex algebraic geometry}, SIAM, 2012.

\bibitem{chen2009tensor}
Jie Chen and Yousef Saad, \emph{On the tensor svd and the optimal low rank
  orthogonal approximation of tensors}, SIAM J. Matrix Anal. Appl. \textbf{30}
  (2009), no.~4, 1709--1734.

\bibitem{comon1994independent}
Pierre Comon, \emph{Independent component analysis, a new concept?}, Signal
  processing \textbf{36} (1994), no.~3, 287--314.

\bibitem{comon2008symmetric}
Pierre Comon, Gene Golub, Lek-Heng Lim, and Bernard Mourrain, \emph{Symmetric
  tensors and symmetric tensor rank}, SIAM J. Matrix Anal. Appl. \textbf{30}
  (2008), no.~3, 1254--1279.

\bibitem{de2008tensor}
Vin De~Silva and Lek-Heng Lim, \emph{{Tensor rank and the ill-posedness of the
  best low-rank approximation problem}}, SIAM J. Matrix Anal. Appl. \textbf{30}
  (2008), no.~3, 1084--1127.

\bibitem{friedland2013best}
Shmuel Friedland, \emph{Best rank one approximation of real symmetric tensors
  can be chosen symmetric}, Front. Math. China \textbf{8} (2013), no.~1,
  19--40.

\bibitem{friedland2016remarks}
\bysame, \emph{Remarks on the symmetric rank of symmetric tensors}, SIAM J.
  Matrix Anal. Appl. \textbf{37} (2016), no.~1, 320--337.

\bibitem{friedland2018nuclear}
Shmuel Friedland and Lek-Heng Lim, \emph{{Nuclear norm of higher-order
  tensors}}, Math. Comput. \textbf{87} (2018), no.~311, 1255--1281.

\bibitem{friedland2018spectral}
Shmuel Friedland and Li~Wang, \emph{Spectral norm of a symmetric tensor and its
  computation}, arXiv preprint arXiv:1808.03864 (2018).

\bibitem{hackbusch2012tensor}
Wolfgang Hackbusch, \emph{{Tensor spaces and numerical tensor calculus}},
  vol.~42, Springer Science \& Business Media, 2012.

\bibitem{henrion2009gloptipoly}
Didier Henrion, Jean-Bernard Lasserre, and Johan L{\"o}fberg, \emph{{GloptiPoly
  3: moments, optimization and semidefinite programming}}, Optim. Methods
  Softw. \textbf{24} (2009), no.~4-5, 761--779.

\bibitem{hillar2013most}
Christopher~J Hillar and Lek-Heng Lim, \emph{{Most tensor problems are
  NP-hard}}, J. ACM \textbf{60} (2013), no.~6, 45.

\bibitem{kolda2001orthogonal}
Tamara~G Kolda, \emph{{Orthogonal tensor decompositions}}, SIAM J. Matrix Anal.
  Appl. \textbf{23} (2001), no.~1, 243--255.

\bibitem{kolda2009tensor}
Tamara~G Kolda and Brett~W Bader, \emph{{Tensor decompositions and
  applications}}, SIAM Rev. \textbf{51} (2009), no.~3, 455--500.

\bibitem{lim2006singular}
Lek-Heng Lim, \emph{Singular values and eigenvalues of tensors: a variational
  approach}, arXiv preprint math/0607648 (2006).

\bibitem{lim2014blind}
Lek-Heng Lim and Pierre Comon, \emph{Blind multilinear identification}, IEEE
  Trans. Inf. Theory \textbf{60} (2014), no.~2, 1260--1280.

\bibitem{Lofberg2004}
J.~L{\"{o}}fberg, \emph{{YALMIP : A Toolbox for Modeling and Optimization in
  MATLAB}}, In Proceedings of the CACSD Conference (Taipei, Taiwan), 2004.

\bibitem{mu2015successive}
Cun Mu, Daniel Hsu, and Donald Goldfarb, \emph{Successive rank-one
  approximations for nearly orthogonally decomposable symmetric tensors}, SIAM
  J. Matrix Anal. Appl. \textbf{36} (2015), no.~4, 1638--1659.

\bibitem{mu2017greedy}
\bysame, \emph{Greedy approaches to symmetric orthogonal tensor decomposition},
  SIAM J. Matrix Anal. Appl. \textbf{38} (2017), no.~4, 1210--1226.

\bibitem{nie2014semidefinite}
Jiawang Nie and Li~Wang, \emph{Semidefinite relaxations for best rank-1 tensor
  approximations}, SIAM J. Matrix Anal. Appl. \textbf{35} (2014), no.~3,
  1155--1179.

\bibitem{qi2005eigenvalues}
Liqun Qi, \emph{Eigenvalues of a real supersymmetric tensor}, J. Symb. Comput.
  \textbf{40} (2005), no.~6, 1302--1324.

\bibitem{schatz2014exploiting}
Martin~D Schatz, Tze~Meng Low, Robert~A van~de Geijn, and Tamara~G Kolda,
  \emph{Exploiting symmetry in tensors for high performance: Multiplication
  with symmetric tensors}, SIAM J. Sci. Comput. \textbf{36} (2014), no.~5,
  C453--C479.

\bibitem{shitov2018counterexample}
Yaroslav Shitov, \emph{A counterexample to comon's conjecture}, SIAM J. Appl.
  Algebra Geom. \textbf{2} (2018), no.~3, 428--443.

\bibitem{solomonik2016contracting}
Edgar Solomonik and James~W Demmel, \emph{Contracting symmetric tensors using
  fewer multiplications}, Linear Algebra Appl. (2016).

\bibitem{sorensen2012canonical}
Mikael S{\o}rensen, Lieven~De Lathauwer, Pierre Comon, Sylvie Icart, and Luc
  Deneire, \emph{Canonical polyadic decomposition with a columnwise orthonormal
  factor matrix}, SIAM Journal on Matrix Analysis and Applications \textbf{33}
  (2012), no.~4, 1190--1213.

\bibitem{stegeman2010subtracting}
Alwin Stegeman and Pierre Comon, \emph{Subtracting a best rank-1 approximation
  does not necessarily decrease tensor rank}, Linear Algebra Appl. \textbf{433}
  (2010), no.~7, 1276--1300.

\bibitem{wang2015orthogonal}
Liqi Wang, Moody~T Chu, and Bo~Yu, \emph{Orthogonal low rank tensor
  approximation: Alternating least squares method and its global convergence},
  SIAM J. Matrix Anal. Appl. \textbf{36} (2015), no.~1, 1--19.

\bibitem{yang2015sdpnal}
Liuqin Yang, Defeng Sun, and Kim-Chuan Toh, \emph{{SDPNAL$+$: a majorized
  semismooth Newton-CG augmented Lagrangian method for semidefinite programming
  with nonnegative constraints}}, Math. Program. Comput. \textbf{7} (2015),
  no.~3, 331--366.

\bibitem{zhang2001rank}
Tong Zhang and Gene~H Golub, \emph{{Rank-one approximation to high order
  tensors}}, SIAM J. Matrix Anal. Appl. \textbf{23} (2001), no.~2, 534--550.

\bibitem{zhang2016comon}
Xinzhen Zhang, Zheng-Hai Huang, and Liqun Qi, \emph{Comon's conjecture, rank
  decomposition, and symmetric rank decomposition of symmetric tensors}, SIAM
  J. Matrix Anal. Appl. \textbf{37} (2016), no.~4, 1719--1728.

\bibitem{zhang2012best}
Xinzhen Zhang, Chen Ling, and Liqun Qi, \emph{The best rank-1 approximation of
  a symmetric tensor and related spherical optimization problems}, SIAM J.
  Matrix Anal. Appl. \textbf{33} (2012), no.~3, 806--821.

\bibitem{zhang2012cubic}
Xinzhen Zhang, Liqun Qi, and Yinyu Ye, \emph{The cubic spherical optimization
  problems}, Math. Comput. \textbf{81} (2012), no.~279, 1513--1525.

\end{thebibliography}

\end{document}